\documentclass[12pt]{amsart}

\usepackage[top=0.9in, bottom=1in, left=1in, right=1in]{geometry}
\usepackage{amsmath, amssymb, amsthm}
\usepackage{verbatim}
\usepackage{graphicx}
\usepackage{enumerate}
\usepackage{mathrsfs}
\usepackage{hyperref}
\usepackage{makecell}
\usepackage{verbatim}
\usepackage{tikz-cd}

\newtheorem{thm}{Theorem}[section]
\newtheorem{prop}[thm]{Proposition}
\newtheorem{lem}[thm]{Lemma}
\newtheorem{cor}[thm]{Corollary}

     \newcommand{\sA}{\mathcal A}

     \newcommand{\sE}{\mathcal E}

\newcommand{\fM}{\mathfrak M}     \newcommand{\sM}{\mathcal M}
     
     \newcommand{\sO}{\mathcal O}

\newcommand{\fR}{\mathfrak R}     \newcommand{\sR}{\mathcal R}
     \newcommand{\sS}{\mathcal S}
     
\newcommand{\fU}{\mathfrak U}

\newcommand{\fr}{\mathfrak r}
\newcommand{\afr}{\mathscr{R}_{\textrm{aff}}}
\newcommand{\afrsum}{\; \hat + \;}

\newcommand{\afrprod}{\; \hat \cdot \;}

\newcommand{\afc}{\mathscr{C}_{\textrm{aff}}}

\def\XXint#1#2#3{{\setbox0=\hbox{$#1{#2#3}{\int}$ }
\vcenter{\hbox{$#2#3$ }}\kern-.6\wd0}}

\newcommand{\R}{\mathbb{R}}
\newcommand{\Z}{\mathbb{Z}}
\newcommand{\C}{\mathbb{C}}

\newcommand{\N}{\mathbb{N}}

\newcommand{\K}{\mathbb{K}}

\newcommand{\iu}{{i\mkern1mu}}

\theoremstyle{definition}
\newtheorem{definition}[thm]{Definition}

\newtheorem{notation}[thm]{Notation}

\newtheorem{example}[thm]{Example}

\theoremstyle{remark}

\newtheorem{remark}[thm]{Remark}

\numberwithin{equation}{section}

\newtheoremstyle{ser}
{8pt}
{8pt}
{\it}
{}
{\sf}
{:}
{6mm}
{}

\newtheoremstyle{serr}
{8pt}
{8pt}
{\normalfont}
{}
{\sf}
{.}
{6mm}
{}

\theoremstyle{ser}

\theoremstyle{serr}

\theoremstyle{ser}
\newtheorem{prob}{Problem}

\begin{document}


\title[A Framework for Rank Identities]{A framework for rank identities - with a view towards operator algebras}


\author{Soumyashant Nayak}

\address{Indian Statistical Institute\\
 Statistics and Mathematics Unit\\
 8th Mile, Mysore Road\\
 Bengaluru -- 560 059, Karnataka, India.
   ORCiD: 0000-0002-6643-6574}
\email{soumyashant@gmail.com}
\urladdr{https://www.isibang.ac.in/$\sim$soumyashant/} 




\maketitle

\begin{abstract}
For every square matrix $A$ over a field $\mathbb{K}$, we have the equality $\mathrm{rank}(A) + \mathrm{rank}(I-A) = \mathrm{rank}(I) + \mathrm{rank}(A-A^2)$ where $I$ denotes the identity matrix with the same dimensions as $A$. In this article, we start a program to systematically characterize and generalize such rank identities with a view towards applications to operator algebras. We initiate the study of so-called ranked rings (unital rings with a `rank-system'), the main examples of interest being finite von Neumann algebras, Murray-von Neumann algebras, and von Neumann rank-rings. In our framework, a field $\mathbb{K}$ may be viewed as a ranked ring with a $\mathbb{Z}^+$-valued rank-system consisting of the usual rank functions on $M_n(\mathbb{K})$ (for $n \in \mathbb{N})$ and serves as the motivating example. We show that a finite von Neumann algebra $\mathscr{R}$ with center $\mathscr{C}$ (and the corresponding Murray-von Neumann algebra $\afr$) may be endowed with a $\mathscr{C}^{+}$-valued rank-system, considering $\mathscr{C}^{+}$ as a commutative monoid with respect to operator addition. We give an algorithm to generate rank identities in ranked $\mathbb{K}$-algebras via the polynomial function calculus, and in ranked complex Banach algebras via the holomorphic function calculus. As an illustrative application, using these abstract rank identities we show that the sum of finitely many idempotents $e_1, \ldots, e_m$ in a finite von Neumann algebra is an idempotent if and only if they are mutually orthogonal, that is, $e_i e_j = \delta_{ij} e_i$ for $1 \le i, j \le m$. In contrast, this does not hold in general in the ring of bounded operators on an infinite-dimensional complex Hilbert space. 

\vskip 0.01in \noindent
{\bf MSC2010 subject classification:} 13G05, 15A03, 46L10, 47L60
\end{abstract}






\section{Introduction}

In the development of the theory of operator algebras, the self-adjoint algebras (such as $C^*$-algebras, von Neumann algebras) have played an instrumental role, with far-reaching implications. Even for the study of a single operator of interest, the study of algebras of operators (containing the operator) provides valuable insight.  For instance, the natural setting to discuss the spectral decomposition of a normal operator is the smallest von Neumann algebra generated by the operator. In the simpler case of a linear transformation $T$ acting on a finite-dimensional complex vector space, a thorough understanding of $T$ can be achieved through the study of the algebra of polynomials in $T$. (Note that the adjoint operation, $*$, does not play a role here). One may represent $T$ in a Jordan canonical form in a suitably chosen basis, which makes the fundamental behaviour of the linear transformation apparent. Unfortunately, such a lucid description is unavailable in general for operators acting on an infinite-dimensional Hilbert space owing to our meagre understanding of reducibility properties of arbitrary operators; a fundamental roadblock is a deep unsolved problem in operator theory commonly known as the {\it invariant subspace problem}. For this reason, we deem results that hold for arbitrary operators (without assumptions such as self-adjointness, normality, etc.) to be invaluable.

In this article, we explore some algebraic aspects of finite von Neumann algebras, with copious motivation from the study of linear transformations on finite dimensional spaces. Our discussion treats algebras of matrices over a field and finite von Neumann algebras on an equal footing illuminating the relationship between the two subjects. Many of the results we obtain are novel even in the finite dimensional case. 

The rank of a linear transformation may be viewed as a measure of its non-degeneracy. In a sense, the rank and nullity are the only invariants of a linear transformation under any change of basis in the domain or the target vector space. Identities involving ranks of various linear transformations carry information about the subtle relationship between the `degeneracy' of these transformations. For a square matrix $A$ with entries from a field $\K$, we have the equality 
\begin{equation}
\label{eqn:rank_id}
\mathrm{rank}(A) + \mathrm{rank}(I - A) = \mathrm{rank}(I) + \mathrm{rank}(A-A^2).
\end{equation}
(Here $I$ stands for the identity matrix with same dimensions as $A$). One easily concludes from the sub-additivity of rank that 
\begin{equation}
\label{eqn:subadd_rank}
\mathrm{rank}(I) \le \mathrm{rank}(A) + \mathrm{rank}(I-A)
\end{equation}
with equality if and only if $A$ is an idempotent. In fact, we have the more general result.
\begin{prop}[Cochran's Theorem]
\label{prop:matrix_rank_cochran}
\textsl{
For a field $\K$, let $A_1, \ldots, A_m$ be matrices in $M_n(\K)$ such that $\sum_{i=1}^m A_i = I$. If $\sum_{i=1}^m \mathrm{rank} (A_i) = n$, then the $A_i$'s are mutually orthogonal idempotents, that is, $A_i A_j = \delta_{ij}A_j$ for $1 \le i, j \le m$.
}
\end{prop}
\noindent Loosely, we may interpret the above result as describing a necessary and sufficient condition (based on the ranks), for a form of independence of the matrices that form a `partition' of the identity matrix. 

In this article, we transplant these ideas and some of their consequences to the setting of finite von Neumann algebras and the corresponding $*$-algebras of affiliated operators (Murray-von Neumann algebras). Although the center-valued dimension function makes it possible to define a notion of rank of an operator in a finite von Neumann algebra and its corresponding Murray-von Neumann algebra, a careful study seems to have been neglected in the literature. Before we dive into the details, we outline the main goals of this article, which are:
\begin{itemize}
\item[(i)] To develop a framework to systematically study rank identities via an algebraic object (which we call the $\fU \fR$-monoid) associated with a unital ring, the rings of primary interest being polynomial rings over fields, free associative algebras over fields, the ring of entire functions on $\C$, and their quotients with (two-sided) ideals generated by relations of interest; 
\item[(ii)] To explore properties of the center-valued rank on finite von Neumann algebras and corresponding Murray-von Neumann algebras;
\item[(iii)] To derive several useful abstract rank identities involving idempotents (with applications in the context of projections in a finite von Neumann algebra);
\item[(iv)] To apply the results so obtained to gain some insight into the structure of finite von Neumann algebras and Murray-von Neumann algebras.
\end{itemize}
\noindent In the next section, we give a brief overview of the results in this article which strive towards the above goals.

\section{Overview of Results}
The key notion introduced in this article is that of a {\it ranked ring}, which is a ring $\mathcal{R}$ equipped with a rank-system, $\rho : \sqcup_{n \in \N} M_n(\sR) \to \sM$, which is a mapping from the set of square matrices over $\sR$ to a commutative monoid $\sM$. The formal definition of a rank-system is given in \S  \ref{sec:ranked_rings}.\ The main property of our abstract system of rank functions is motivated by the invariance of the rank under change of basis, which is codified by invariance of $\rho$ under left multiplication or right multiplication by invertible matrices. In Theorem \ref{thm:univ_rank_sys}, we show that every ring $\sR$ comes equipped with a universal rank-system which serves as the fountainhead for all rank-systems on $\sR$. A field $\K$ may be viewed as a ranked ring with a $\mathbb{Z}^+$-valued rank-system ($\mathbb{Z}^+ := \mathbb{N} \cup \{ 0 \}$) consisting of the usual rank functions on $M_n(\K)$ (for $n \in \mathbb{N})$ and serves as the motivating example.

A few words of explanation are in order regarding this level of generality. The techniques that are used in the literature (cf.\ \cite{tian-styan1}, \cite{tian-styan2}) to obtain rank identities use little beyond the additive structure of $\mathbb{Z}^+$ and thus may be gainfully transferred to cases where the values taken by the rank-system are in a commutative monoid. This generality comes in especially handy after we exhibit a center-valued rank-system for finite von Neumann algebras and the corresponding Murray-von Neumann algebras (see \S \ref{subsec:center_valued}), and also in the context of $\R^{+}$-valued rank-systems on von Neumann regular rings. For von Neumann regular rings, we show that there is a natural one-to-one correspondence between $\R^{+}$-valued rank-systems and pseudo-rank functions (as defined in \cite{goodearl}); see Remark \ref{rmrk:rank_sys_ring}.

In Algorithm \ref{subsec:algo_rank_id}, we give an explicit algorithm for generating rank identities similar to identity (\ref{eqn:rank_id}) involving polynomials in a single element in a ranked $\K$-algebra. In particular, these rank identities hold for all square matrices over $\K$. The algorithm uses Theorem \ref{thm:rank_id_pol} mentioned below and the fact that the univariate polynomial ring, $\K[t]$, is a principal ideal domain. Furthermore, the list of rank identities obtained is exhaustive in a generic sense (see Proposition \ref{prop:exhaustive}). In Theorem \ref{thm:rank_id_pol}, we consider the equivalence classes (denoted $[\cdot]$) of associates in GCD domains which naturally form a distributive lattice called the divisibility lattice, with the join ($\vee$) given by the GCD and the meet ($\wedge$) given by the LCM.
\vskip 0.1in

\noindent {\bf Theorem \ref{thm:rank_id_pol} }
Let $\mathcal{R}$ be a $(\sM, {\it rk})$-ranked ring and $\K$ be a subfield of the center of $\sR$. For non-zero polynomials $p_1, \ldots, p_n$, $q_1, \ldots, q_n$ in $\K[t]$, if
\begin{equation}
\wedge_{1\le i_1 < \cdots < i_k  \le n} ([p_{i_1}] \vee \cdots \vee [p_{i_k}]) = \wedge_{1 \le i_1 < \cdots < i_k \le n} ([q_{i_1}] \vee \cdots \vee [q_{i_k}] ), 1 \le k \le n,
\end{equation}
then for $x$ in $\mathcal{R}$, we have
$$\sum_{i=1}^n {\it rk}(p_i(x)) = \sum_{i=1}^n {\it rk}(q_i(x)).$$
\vskip 0.1in

Using the results discussed above, we resolve a question posed in \cite[pg.\ 104, 2.8.3]{the_rota_way}. Reflecting on the Frobenius rank inequality, $$\mathrm{rank}(AB) + \mathrm{rank}(BC) \le \mathrm{rank}(B) + \mathrm{rank}(ABC), \textrm{ for }A, B, C \in M_n(\K),$$ the authors of \cite{the_rota_way} wonder about a theory of rank inequalities - ``... are they all consequences of a finite set of inequalities?" As observed in (\ref{eqn:subadd_rank}), rank inequalities may be obtained from rank identities by removing some terms from one side of the equality. Algorithm \ref{subsec:algo_rank_id} in this article yields infinitely many fundamentally distinct rank identities in terms of polynomial expressions involving a single matrix in $M_n(\K)$. Hence we conclude that the answer to the above question is negative.

In a similar vein, we prove the theorem below characterizing rank identities in ranked complex Banach algebras involving the holomorphic function calculus. Here $\sO (\C)$ denotes the ring of entire functions on $\C$.
\vskip 0.1in

\noindent {\bf Theorem \ref{thm:rank_entire_id} }
Let $\mathcal{R}$ be a $(\sM, {\it rk})$-ranked complex Banach algebra. For entire functions $f_1, \ldots, f_n$, $g_1, \ldots, g_n$ in $\mathcal{O}(\mathbb{C})$ if 
\begin{equation}
\wedge_{1\le i_1\le \cdots \le i_k  \le n} ([f_{i_1}] \vee \cdots \vee [f_{i_k}]) = \wedge_{1\le i_1\le \cdots \le i_k  \le n} ([g_{i_1}] \vee \cdots \vee [g_{i_k}] ), 1 \le k \le n,
\end{equation}
then for $x$ in $\mathcal{R}$, we have
$$\sum_{i=1}^n {\it rk}(f_i(x)) = \sum_{i=1}^n {\it rk}(g_i(x)).$$
\vskip 0.1in

With the goal of systematizing the study of abstract rank identities and {\it en route} to Algorithm \ref{subsec:algo_rank_id} mentioned above, in \S \ref{sec:l_functor} we construct a covariant functor $\fU \fR$ from the category of (unital) rings to the category of commutative monoids. The $\fU \fR$-monoid of a ring helps equip the ring with a rank-system which is universal in the appropriate sense (see Theorem \ref{thm:univ_rank_sys}). Although the $\fU \fR$-monoid is defined for all rings, for the purpose of studying rank identities, the rings of primary interest are formal algebras such as polynomial rings over a field ($\K[t_1, \ldots, t_n]$), free associative algebras over a field ($\K \langle t_1, \ldots, t_n \rangle$), etc. Each element of the $\fU \fR$-monoid captures a class of rank identities. The $\fU \fR$-monoid of quotient rings of these formal rings with (two-sided) ideals generated by relations of interest is also important as evidenced by the discussion on rank identities involving idempotents (see \S \ref{subsec:rank_id_idem}) and subsequent applications in \S \ref{sec:app}. In this context, the ring of interest is $\K\langle t_1, t_2 \rangle / (t_1 - t_1^2, t_2 - t_2^2)$ where $(t_1 - t_1^2, t_2 - t_2^2)$ is the two-sided ideal of $\K \langle t_1, t_2 \rangle$ generated by  the non-commutative polynomials $t_1 - t_1^2$ and $t_2 - t_2^2$.

In \S \ref{sec:L_bezout}, we characterize the $\fU \fR$-monoid of elementary divisor domains, and Murray-von Neumann algebras. Keeping in mind that elementary divisor domains are B\'{e}zout domains and thus GCD-domains, we completely characterize the $\fU \fR$-monoid of elementary divisor domains in terms of multichains in the divisibility lattice.\ For the elementary divisor domains $\K[t]$ and $\mathcal{O}(\C)$, we use this characterization to prove the previously mentioned key results, Theorem \ref{thm:rank_id_pol}, Theorem \ref{thm:rank_entire_id}.

\subsection{Applications to Operator Algebras}
\label{subsec:op_alg_res_overview}
We briefly examine how the above discussion relates to operator algebras. Let $\mathscr{R}$ be a finite von Neumann algebra acting on the Hilbert space $\mathscr{H}$ with center $\mathscr{C}$. Let $\afr$ denote the Murray-von Neumann algebra of closed densely-defined operators affiliated with $\mathscr{R}$. Using the canonical $\mathscr{C}$-valued trace on $\mathscr{R}$, we define the rank (denoted by $\fr _c$) of an operator in $\afr$ as the trace of its range projection, which is the $\mathscr{C}^{+}$-valued dimension of its range projection; note that the range projection of an operator in $\afr$ is a projection in $\mathscr{R}$. In the context of $\mathscr{R}$, this definition may also be found in \cite{tr_rank} by the name of tr-rank.

In \cite{nayak_matrix_mvna}, the present author has obtained an abstract description of Murray-von Neumann algebras independent of their representation and shown that $M_n(\afr) \cong M_n(\mathscr{R})_{\textrm{aff}}$ as unital ordered complex topological $*$-algebras (see \cite[Theorem 4.14]{nayak_mvna}).  With this isomorphism at hand, we prove in Theorem \ref{thm:vna_rank_system} that the rank functions (scaled by $n$) on $M_n(\afr)$ define a $\mathscr{C}^{+}$-valued rank-system for $\afr$, where $\mathscr{C}^{+}$ is considered as a commutative monoid with respect to operator addition. Thus the results proved for ranked $\C$-algebras are applicable in this setting. Furthermore, it is clear that the restriction of the rank-system to $\mathscr{R}$ not only makes $\mathscr{R}$ a ranked $\C$-algebra but also a ranked complex Banach algebra. 
 
We prove several rank identities for idempotents in abstract ranked rings that are useful in establishing the basic properties of the center-valued rank $\fr _c$. For projections $e, f$ in $\mathscr{R}$, the identity, 
\begin{align*}
\fr _c(e+f) &= \fr _c \big( (1-f)e(1-f) \big) + \fr _c(f)\\
 &= \fr _c(e) + \fr_c \big( (1-e)f(1-e) \big),
\end{align*} helps in establishing the subadditivity of rank, that is, $\fr _c(x+y) \le \fr _c(x) + \fr _c(y)$ for $x, y \in \afr$, and naturally begets the necessary and sufficient conditions for equality (see Theorem \ref{thm:vna_rank_sub}).

As an application, we show a generalized version of Proposition \ref{prop:matrix_rank_cochran} in the setting of von Neumann regular rings and, as a corollary, in the setting of Murray-von von Neumann algebras.
\vskip 0.1in

\noindent {\bf Corollary \ref{cor:cochran_vNreg} }
\textsl{
Let $\mathscr{R}$ be a finite von Neumann algebra. Let $a_1, \ldots, a_n$ be elements of $\afr$ and $e$ be an idempotent in $\afr$ such that $\sum_{i=1}^n a_i = e$, and $\sum_{i=1}^n \fr _c(a_i) = \fr _c(e)$. Then the $a_i$'s are mutually orthogonal idempotents, that is, $a_i a_j  = \delta_{ij}a_i, 1\le i, j \le n$.
}
\vskip 0.1in

Using the fact that the rank of an idempotent in $\mathscr{R}$ is equal to the trace of the idempotent, we prove Theorem \ref{thm:sum_of_idem_fva} mentioned below. Viewed in isolation, note that the statement of Theorem \ref{thm:sum_of_idem_fva} does not involve the rank $\fr _c$.
\vskip 0.1in

\noindent {\bf Theorem \ref{thm:sum_of_idem_fva} } 
\textsl{
Let $\mathscr{R}$ be a finite von Neumann algebra. For idempotents $e_1, e_2, \ldots, e_n$ in $\mathscr{R}$, the operator $e_1 + e_2 + \cdots + e_n$ is an idempotent if and only if the $e_i$'s are mutually orthogonal, that is, $e_i e_j = \delta_{ij} e_i, 1\le i, j \le n$.
}
\vskip 0.1in

The above result is not true in general for von Neumann algebras that are not finite. In \cite[Example 3.1]{bart_zero_sum}, an example is shown of five idempotents acting on an infinite-dimensional separable complex Hilbert space that are not mutually orthogonal but whose sum is equal to $0$, which is an idempotent. If we replace `idempotents' with `projections', which are self-adjoint idempotents, the result is well-known to hold in all $C^*$-algebras.

In \cite[\S 5]{nayak_matrix_mvna}, as an application of rank identities, the present author has shown that if $p, q$ in $\afr$ satisfy the Heisenberg commutation relation, $qp-pq = i1$, then the respective point spectra of $p$ and $q$ must be empty, that is, $p$ and $q$ must be invertible elements in $\afr$. The existence of such a pair of elements in $\afr$ is, to date, an open question. This conundrum (the so-called Heisenberg-von Neumann puzzle) may be recast in the following manner - Are there invertible operators $p, a$ in $\afr$ such that $p^{-1}ap = 1+a$? Or equivalently, is there an invertible operator $a$ in $\afr$ such that $a$ and $1+a$ belong to the same conjugacy class of $\afr$? This suggests that any strategy towards its resolution must involve the study of similarity orbits (and related invariants) in $\afr$ in an essential way.

\subsection{Organization of the article}
In \S \ref{sec:prelim}, we briefly review basic concepts and results from multiset theory, ring theory and the theory of von Neumann algebras that are relevant to our dicussion. In \S \ref{sec:lattice}, we discuss basic concepts in lattice theory as preparation for our applications in the context of the divisibility lattice of elementary divisor domains. In \S \ref{sec:l_functor}, we define the functor $\fU \fR$ from the category of rings to the category of commutative monoids.\ In \S \ref{sec:ranked_rings}, we define the notion of rank-systems and ranked rings, and study some examples. In \S \ref{sec:L_bezout}, we completely characterize the $\fU \fR$-monoids of elementary divisor domains and Murray-von Neumann algebras. In \S \ref{sec:rank_id}, we prove several abstract rank identities in ranked rings, many of which involve idempotents. In the final section \S \ref{sec:app}, we discuss some applications of the theory developed.

\section{Preliminaries}
\label{sec:prelim}

In this section, we set up the basic notation to be used throughout the article and briefly review the relevant concepts and results from multiset theory, ring theory and the theory of von Neumann algebras. 

\begin{notation}
We use the following notation.\\
$\mathbb{N} := $ set of natural numbers,\\
$\mathbb{Z} := $ set of integers,\\
$\mathbb{R} := $ set of real numbers,\\
$\mathbb{C} := $ set of complex numbers,\\
$\mathbb{Z}^+ := \mathbb{N} \cup \{ 0 \}$,\\
$\R ^{+} := $ set of non-negative real numbers,\\
For $n \in \mathbb{N}$, $\langle n \rangle := \{1, 2, \ldots, n \}$.

We often use $\sR$ to denote a ring, and $\sM$ to denote a commutative monoid. The ring of $n \times n$ matrices over $\sR$ is denoted by $M_n(\sR)$. The identity matrix in $M_n(\sR)$ is denoted by ${\bf 1}_n$ and the zero matrix is denoted by ${\bf 0}_n$. For matrices $X \in M_k(\sR), Y \in M_{\ell}(\sR)$, we define $X \odot Y := \mathrm{diag}(X, Y)$ in $M_{k+\ell}(\sR)$. The set of finite square matrices with entries from $\sR$ is denoted by $M_{\mathrm{fin}}(\sR) := \sqcup_{i \in \mathbb{N} } M_i(\sR)$.

We use $\K$ to denote a field. The polynomial ring in one variable over $\K$ is denoted by $\K[t]$, and the free associative algebra over $\K$ with $n$ generators is denoted by $\K \langle t_1, \ldots, t_n \rangle$. The ring of entire functions on $\C$ is denoted by $\sO (\C)$.
\end{notation}

\subsection{Multiset Theory}

A {\it multiset} is a collection of objects in which elements may occur more than once but finitely many times. In \cite[\S 3]	{blizard}, multisets are modeled by $\mathbb{N}$-valued functions on sets counting the multiplicity of each element. We use the language of multisets because of the convenience it affords. For instance, by the fundamental theorem of algebra, the zeroes of a polynomial of degree $d$ over $\mathbb{C}$ form a multiset of cardinality $d$. The primes in the prime factorization of a natural number ($n = p_1 ^{\alpha_1} \cdots p_k ^{\alpha_k}$) form a multiset of cardinality $\alpha_1 + \cdots + \alpha_k$. 

The underlying set of a multiset $A$ is called the {\it support} of $A$ and denoted by $\mathscr{S}(A)$. Let $A, B$ be multisets with multiplicity functions $m_A, m_B$.
\begin{itemize}
\item[(i)] (Inclusion) $A$ is a multisubset of $B$ if $\mathscr{S}(A) \subseteq \mathscr{S}(B)$ and $m_A(x) \le m_B(x)$ for $x \in \mathscr{S}(A)$.
\item[(ii)] (Intersection) The intersection of $A$ and $B$ is the multiset $C$ such that $\mathscr{S}(C) = \mathscr{S}(A) \cap \mathscr{S}(B)$ and $m_C(x) = \min \{ m_A(x), m_B(x) \}$ for $x \in \mathscr{S}(C).$
\item[(iii)] (Union) The union of $A$ and $B$ is the multiset $C$ such that $\mathscr{S}(C) = \mathscr{S}(A) \cup \mathscr{S}(B)$ and $m_C(x) = \max \{ m_A(x), m_B(x) \}$ for $x \in \mathscr{S}(C)$ where it is understood that $m_A(x) = 0$ if $x \notin \mathscr{S}(A)$ and $m_B(x) = 0$ if $x \notin \mathscr{S}(B)$. 
\item[(iv)] (Sum) The sum of $A$ and $B$ is the multiset $C$ such that $\mathscr{S}(C) = \mathscr{S}(A) \cup \mathscr{S}(B)$ and $m_C(x) = m_A(x) +  m_B(x)$ for $x \in \mathscr{S}(C)$ where it is understood that $m_A(x) = 0$ if $x \notin \mathscr{S}(A)$ and $m_B(x) = 0$ if $x \notin \mathscr{S}(B)$. 
\end{itemize}
The sum of multisets may be viewed as the multiset version of the notion of disjoint union for sets.

\subsection{Rings}
For us, rings always contain a multiplicative identity, denoted by $1$. 

\begin{definition}
Let $\mathcal{R}$ be a ring. An element $u$ of $\mathcal{R}$ is said to be {\it left-invertible} if there is an element $v$ of $\mathcal{R}$ such that $vu = 1$. Similarly if there is a $v$ in $\mathcal{R}$ such that $uv = 1$, then $u$ is said to be {\it right-invertible}. An element of $\mathcal{R}$ which is both left-invertible and right-invertible is said to be {\it invertible}. 
\end{definition}

An invertible element in a ring has a unique left-inverse and a unique right-inverse which are identical. In summary, an invertible element has a unique inverse.

\begin{lem}
\label{lem:invert_matrix}
\textsl{
If $\mathcal{R}$ is commutative, a matrix $A$ in $M_n(\mathcal{R})$ ($n \in \mathbb{N}$) is invertible if and only if $\det (A)$ is invertible in $\mathcal{R}$.}
\end{lem}

\begin{lem}
\label{lem:invertible_M2}
For $x, y$ in $\mathcal{R}$, we have 
\begin{itemize}
\item[(i)] $\begin{bmatrix}
1 & x\\
0 & 1
\end{bmatrix}^{-1} = 
\begin{bmatrix}
1 & -x\\
0 & 1
\end{bmatrix}$, $\begin{bmatrix}
1 & 0\\
x & 1
\end{bmatrix}^{-1} = 
\begin{bmatrix}
1 & 0\\
-x & 1
\end{bmatrix},$
\item[(ii)] $\begin{bmatrix}
x & 1\\
1 & 0
\end{bmatrix}^{-1} = 
\begin{bmatrix}
0 & 1\\
1 & -x
\end{bmatrix}$, $\begin{bmatrix}
0 & 1\\
1 & x
\end{bmatrix}^{-1} = 
\begin{bmatrix}
-x & 1\\
1 & 0
\end{bmatrix},$
\item[(iii)] $\begin{bmatrix}
x & -1 \\
1 & 0
\end{bmatrix}
\begin{bmatrix}
1 & y \\
x & 0
\end{bmatrix}
\begin{bmatrix}
-y & 1\\
1 & 0
\end{bmatrix} = 
\begin{bmatrix}
xy & 0\\
0 & 1
\end{bmatrix}.$
\end{itemize}

\end{lem}
\begin{proof}
The results follow from straightforward matrix multiplication computations.
\end{proof}

In this article, our discussion involves various kinds of rings and we note their definitions below for quick reference.

\begin{definition}
\begin{itemize}
\item[(i)] An {\it integral domain} is a commutative ring in which the product of two non-zero elements is non-zero.
\item[(ii)] A {\it GCD domain} is an integral domain in which any two elements have a greatest common divisor. In other words, for $x, y$ in a GCD domain $\mathcal{R}$, there is a $z$ in $\mathcal{R}$ such that for $w \in R$ if $w \mid x, w \mid y$, then $w \mid z$.
\item[(iii)] A {\it B\'e{}zout domain} is an integral domain in which the sum of two principal ideals is a principal ideal.
\item[(iv)] A {\it unique factorization domain} (UFD) is an integral domain in which every non-zero element can be uniquely written as a product of irreducible elements upto order and associates.
\item[(v)] A {\it principal ideal domain}  (PID) is an integral domain in which every ideal is a principal ideal.
\item[(vi)] An integral domain $\sR$ is said to be an {\it elementary divisor domain} (EDD) $\sR$ if for every matrix $A$ (not necessarily square) with entries from $\sR$, there exist invertible square matrices $P$ and $Q$ such that $PAQ$ is a diagonal matrix whose $i^{\textrm{th}}$ diagonal entry divides the $(i+1)^{\textrm{th}}$ diagonal entry.
\end{itemize}
\end{definition}
We have the following class inclusions :\\
PIDs $\subset$ EDDs $\subset$ B\'e{}zout domains $\subset$ GCD domains $\subset$ integral domains $\subset$ commutative rings.

\subsection{Von Neumann regular rings and rank-functions}

In this subsection, we recall some basic concepts and pertinent results from the theory of von Neumann regular rings.

\begin{definition}
A {\it von Neumann regular ring} is a unital ring $\sR$ such that for every element $x \in \sR$, there is an element $y \in \sR$ such that $xyx=x$.
\end{definition}

The ring of affiliated operators corresponding to a finite von Neumann algebra is an important example of a von Neumann regular ring. 

\begin{thm}[{see \cite[Theorem 1.1]{goodearl}}]
\textsl{
A unital ring $\sR$ is von Neumann regular if and only if every principal right (left) ideal of $\sR$ is generated by an idempotent in $\sR$.
}
\end{thm}

\begin{lem}[{see \cite[Lemma 1.6]{goodearl}}]
\textsl{
Let $e_1, \ldots, e_n$ be idempotents in a ring $\sR$ such that $e_1 + \cdots + e_n = 1$. Then $\sR$ is von Neumann regular if and only if for every $x \in e_i \sR e_j$, there is $y \in e_j \sR e_i$ such that $xyx = x$. As a consequence, we have the following assertions.
\begin{itemize}
    \item[(i)] If $\sR$ is a von Neumann regular ring, then for every idempotent $e \in \sR$, the ring $e \sR e$ (with identity $e$) is von Neumann regular.
    \item[(ii)] If $\sR$ is a von Neumann regular ring, then for every $n \in \N$, the matrix ring $M_n(\sR)$ is von Neumann regular.
\end{itemize}
}
\end{lem}

\begin{definition}[{cf. \cite[pg. 231]{goodearl}}]
\label{def:rank-ring}
A {\it rank-function} on a von Neumann regular ring $\sR$ is a map $\fr  : \sR \to \R^{+}$ such that:
\begin{itemize}
\item[(i)] $\fr (1) > 0$;
\item[(ii)] $\fr (xy) \le \fr (x)$, and $\fr (xy) \le \fr (y)$ for all $x, y \in \sR$;
\item[(iii)] $\fr (e+f) = \fr (e) + \fr (f)$ for every pair of mutually orthogonal idempotents $e, f \in \sR$.
\end{itemize}
The rank-function $\fr $ is said to be normalized to $\fr (1)$. If $\fr (x) > 0$ whenever $x \ne 0$, we say that $\fr $ is a {\it non-degenerate} rank-function on $\sR$. 

Note that a non-negative scaling of a rank-function is again a rank-function, and a positive scaling of a non-degenerate rank-function is again a non-degenerate rank-function.
\end{definition}

In the literature, what we have defined as a rank-function is referred to as a {\it pseudo-rank function}, and what we have defined as a nondegenerate rank-function is referred to as a rank-function. Furthermore, it is standard to stipulate that the rank-function be normalized to $1$, that is, $\fr (1) = 1$. Our departure from the convention allows for more flexibility and compatibility with the language used in this article, without affecting the results we invoke from the literature in this context.

In \S \ref{subsec:rank_vNreg}, the theorem mentioned below plays a pivotal role in understanding the natural one-to-one correspondence between rank-functions on $\sR$ and $\R ^{+}$-valued rank-systems on $\sR$.
\begin{thm}[{see \cite[Corollary 16.10]{goodearl}}]
\label{thm:pushout_rank}
\textsl{
Let $\sR$ be a von Neumann regular ring, and $n$ be a positive integer. Let $\delta _n : \sR \to M_n(\sR)$ be the ring homomorphism given by $x \mapsto \odot _{i=1}^n x$. Then the mapping, $\delta _n ^*$, from the set of rank-functions on $M_n(\sR)$ to the set of rank-functions on $\sR$ given by $\delta_n ^*(\fr _n) = \frac{1}{n}(\fr _n \circ \delta _n)$, is bijective.
}
\end{thm}

\subsection{Operator Algebras}
For the discussion on operator algebras below, we primarily follow \cite{kadison-ringrose1}, \cite{kadison-ringrose2}, \cite{nayak_mvna}. Let $\mathscr{H}$ be a complex Hilbert space and $\mathcal{B}(\mathscr{H})$ denote the set of bounded operators on $\mathscr{H}$.

\begin{definition}
A self-adjoint idempotent in $\mathcal{B}(\mathscr{H})$ is said to be a projection. There is a one-to-one correspondence between the set of closed subspaces of $\mathscr{H}$ and the set of projections in $\mathcal{B}(\mathscr{H})$. For an operator $A$ in $\mathcal{B}(\mathscr{H})$, the set $\{ Ax : x \in \mathscr{H} \} \subseteq \mathscr{H}$ is called the {\it range} of $A$. The projection onto the closure of the range of $A$ is said to be the {\it range projection} of $A$ and denoted by $R(A)$. The set $\{ x : Ax = 0, x \in \mathscr{H} \} \subseteq \mathscr{H}$ is called the {\it null space} of $A$. The projection onto the null space is said to be the {\it null projection} of $A$ and denoted by $N(A)$. 
\end{definition}

A $*$-subalgebra of $\mathcal{B}(\mathscr{H})$ containing the identity operator and closed in the weak-operator topology is said to be a {\it von Neumann algebra}. The set of projections in a von Neumann algebra has a natural order structure based on the cone of positive operators. With this order structure, the set of projections is a complete lattice.

\begin{lem}[see {\cite[Proposition 2.5.14]{kadison-ringrose1}}]
\label{lem:range_proj_id}
\textsl{
For projections $E, F$ in $\mathcal{B}(\mathscr{H})$, we have 
\begin{itemize}
\item[(i)] $I - E \wedge F = (I-E) \vee (I-F),$
\item[(ii)] $I - E \vee F = (I-E) \wedge (I-F)$,
\item[(iii)] $R(E+F) = E \vee F,$
\item[(iv)] $R(EF) = E - \big( E \wedge (I-F)\big).$
\end{itemize}
}
\end{lem}

We say that a von Neumann algebra $\mathscr{R}$ is {\it finite} if every isometry in $\mathscr{R}$ is a unitary, that is, if $V^*V = I$ for $V \in \mathscr{R}$, then $VV^* = I$. Finite von Neumann algebras are characterized by the existence of a unique faithful and normal center-valued trace (see \cite[Theorem 8.2.8]{kadison-ringrose2}), which is often referred to as `the' trace on $\mathscr{R}$. 

Let $\mathscr{R}$ be a finite von Neumann algebra. We denote the set of closed densely-defined operators affiliated with $\mathscr{R}$ by $\afr$. A concise account of the theory of unbounded operators may be found in \cite[\S 4]{kadison-liu}. For a more thorough account, the interested reader may refer to \S 2.7, \S 5.6 in \cite{kadison-ringrose1}, or Chapter VIII in \cite{simon-reed}.

From \cite{nayak_mvna}, $\afr$ naturally has the structure of a unital ordered complex topological $*$-algebra and is called the Murray-von Neumann algebra associated with $\mathscr{R}$. We denote the addition in $\afr$ by $\afrsum$, and the multiplication by $\afrprod$. The restriction of $\afrsum, \afrprod$ to $\mathscr{R}$ is the usual operator addition and operator multiplication in $\mathscr{R}$. On the algebraic side, $\afr$ may be viewed as the Ore localization of $\mathscr{R}$ with respect to the multiplicative subset of its non-zero-divisors. When the discussion involves only the algebraic structure of $\afr$, we prefer to use $+, \cdot$, respectively, instead of $\afrsum, \afrprod$, respectively.

\begin{remark}
\label{rmrk:ran_proj}
Let $\mathscr{R}$ be a finite von Neumann algebra acting on the Hilbert space $\mathscr{H}$. For an operator $A$ in $\afr$, the range projection $R(A)$ and the null projection $N(A)$ are both in $\mathscr{R}$. If $A$ is self-adjoint, then $A \afrprod R(A) = R(A) \afrprod A = A.$
\end{remark}

\begin{notation}
We denote  Murray-von Neumann equivalence relation for projections of a von Neumann algebra by $\sim_{MV}$ (see \cite[Chapter 6]{kadison-ringrose2}).
\end{notation}
In the lemma below, we note some results relating the range projection and null projection of various operators in $\afr$ which is useful for our discussion in \S \ref{subsec:center_valued}. 
\begin{lem}[{\cite[Proposition 2.5.13]{kadison-ringrose1} }, {\cite[Proposition 6.1.6]{kadison-ringrose2}}]
\label{lem:range_proj_rel}
\textsl{
Let $\mathscr{R}$ be a finite von Neumann algebra. For operators $A,  B$ in $\afr$, we have
\begin{itemize}
\item[(i)] $R(A) = I- N(A^*)$,
\item[(ii)] $N(A) = I - R(A^*)$,
\item[(iii)] $R(A \afrprod A^*) = R(A) \sim_{MV} R(A^*) = R(A^* \afrprod A)$,
\item[(iv)] $R(A \afrprod B) = R \big( A \afrprod R(B) \big)$,
\item[(v)] $R \big( R(A) \big) = R(A)$.
\end{itemize}
}
\end{lem}

\begin{remark}
\label{rmrk:basis_free_range}
Let $\mathscr{R}$ be a finite von Neumann algebra. For a self-adjoint element $A$ in $\afr$, the range projection of $A$ may be intrinsically defined as the smallest projection $E$ in $\mathscr{R}$ such that $E \afrprod A = A$. This definition is based only on the order and algebraic structure of $\afr$ and is compatible with the definition in the context of its represented form. For an arbitrary operator $A$ in $\afr$, the range projection of $A$ may be defined as the range projection of $A \afrprod A^*$. In \S \ref{subsec:center_valued}, we discuss the notion of center-valued rank on a finite von Neumann algebra and its corresponding Murray-von Neumann algebra and we emphasize here that the results do not depend on any particular representation of the von Neumann algebra.
\end{remark}

\begin{remark}
\label{rmrk:mvna_iso}
Let $\mathscr{R}$ be a finite von Neumann algebra. Recall from \cite[Theorem 4.14]{nayak_mvna} that we have $M_n(\afr) \cong M_n(\mathscr{R})_{\textrm{aff}}$ as unital ordered complex topological $*$-algebras. For the purpose of algebraic computations, we implicitly use the description $M_n(\afr)$ as matrix algebras over $\afr$, whereas while defining the center-valued rank in \S \ref{subsec:center_valued}, it is more convenient to use the operator-theoretic description of $M_n(\mathscr{R})_{\textrm{aff}}$ as affiliated operators.
\end{remark}

\section{Lattice Theoretic Considerations}
\label{sec:lattice}

In this section, we review the basic concepts of lattice theory that are necessary for our discussion. For a detailed account, the reader may consult \cite{birkhoff-latt}. We derive some properties of modular maps from a lattice to an abelian semigroup which serves to provide an economical and insightful language to express our results in \S \ref{sec:L_bezout} in the context of the divisibility lattice associated with an elementary divisor domain.

\subsection{Basic Concepts}

A set $A$ with a partial ordering $\le$ is said to be a {\it poset}. If all elements of $A$ are comparable, then $A$ is said to be a {\it chain}. A multiset with support contained in a poset has an obvious partial order inherited from the poset and is called a multiposet. A multiposet all of whose elements are comparable, that is, whose support is contained in a chain, is said to be a {\it multichain}.

For a subset $H \subseteq A$, we say that $a$ is an upper bound of $H$ if $h \le a$ for all $h \in H$. An upper bound $a$ of $H$ is said to be the supremum of $H$ (denoted $\sup H$) if for any upper bound $b$ of $H$, we have $a \le b$. Note that by the antisymmetry of $\le$ there can at most be one supremum of $H$. 

A poset $(\mathcal{L} ; \le)$ is said to be a {\it lattice} if $\sup \{x, y \}, \inf \{ x, y \}$ exist for all $x, y$ in $\mathcal{L}$. There are two fundamental binary operations on a lattice $\mathcal{L}$ called the {\it join}, which is defined as $x \vee y := \sup \{x, y \}$, and the {\it meet}, which is defined as $x \wedge y := \inf \{ x, y \}$, for $x, y$ in $\mathcal{L}$. The join and the meet are commutative, associative and idempotent. For $x, y $ in $\mathcal{L}$, we say that $x$ {\it covers} $y$ if $y \le x$ and if $y \le a \le x$, then $a = y$ or $a = x$.

\begin{definition}
For lattices $\mathcal{L}_1, \mathcal{L}_2$, a map $\Phi : \mathcal{L}_1 \rightarrow \mathcal{L}_2$ is said to be a {\it lattice homomorphism} if $\Phi(x \vee y) = \Phi(x) \vee \Phi(y), \Phi(x \wedge y) = \Phi(x) \wedge \Phi(y)$ for all $x, y$ in $\mathcal{L}_1$.
\end{definition}

A lattice $\mathcal{L}$ is said to be {\it distributive} if for all $x, y, z \in \mathcal{L}$, we have
\begin{itemize}
\item[(i)] $x \wedge(y \vee z) = (x \vee y) \wedge (x \vee z)$,
\item[(ii)] $x \vee (y \wedge z) = (x \vee  y) \wedge (x \vee z)$.
\end{itemize}

\begin{definition}
Let $X := \{ x_1, \ldots, x_n \}$ be a finite multisubset of $\mathcal{L}$. We define $$X_{(k)} := \wedge_{1 \le i_1 < \cdots < i_k \le n} (x_{i_1} \vee \cdots \vee x_{i_k}), \textrm{ for }1\le k \le n.$$
\end{definition}

\begin{remark}
\label{rmrk:order_stat_latt}
For a multisubset $X$ of cardinality $n$ of a lattice $\mathcal{L}$, we have $X_{(1)} \le X_{(2)} \le \cdots \le X_{(n)}$. If $X$ is a multichain, then $X_{(k)}$ is the $k$-th smallest element in $X$. It is straighforward from the definitions that for a lattice homomorphism $\Phi : \mathcal{L} \rightarrow \mathcal{M}$, we have $\Phi(X)_{(k)} = \Phi(X_{(k)})$ for $1 \le k \le n$.
\end{remark}

\subsection{Modular Maps}

In this subsection, $(\mathcal{L}; \vee, \wedge)$ is a lattice with join $\vee$ and meet $\wedge$. Let $(S; +)$ be an abelian semigroup with the binary operation $+$. 
\begin{definition}
A map $\phi : \mathcal{L} \rightarrow S$ is said to be {\it modular} if $\phi(x) + \phi(y) = \phi(x \wedge y) + \phi(x \vee y)$ for all $x, y $ in $\mathcal{L}$.
\end{definition}

A modular map from $\mathcal{L}$ to $\mathbb{R}$ is commonly known as a {\it valuation} on $\mathcal{L}$ in the lattice theory literature (\cite[pg. 74]{birkhoff-latt}). In Proposition \ref{prop:lat_ord} and Theorem \ref{thm:m_chain_rep}, $\phi$ denotes a modular map from $\mathcal{L}$ to $S$.

\begin{prop}
\label{prop:lat_ord}
\textsl{
Let $x_1 \le \cdots \le x_n$ be a multichain with support contained in $\mathcal{L}$ and $y$ be an element of $\mathcal{L}$. Let $y_1 := y \wedge x_1$, $y_i := (y \vee x_{i-1}) \wedge x_i$ for $2 \le i \le n$, and $y_{n+1} := y \vee x_n$. Then we have
\begin{itemize}
\item[(i)] $y_1 \le y_2 \le \cdots \le y_{n+1} $,
\item[(ii)] $\phi(y) + \sum_{i=1}^n \phi(x_i) = \sum_{i=1}^{n+1} \phi(y_i)$.
\end{itemize} 
}
\end{prop}
\begin{proof}
With notation as in the statement of the theorem, we inductively prove the following assertions for $m \in \langle n \rangle$, $$P(m) : y_1 \le \cdots \le y_m \le y \vee x_m, \textrm{ and } \phi(y) + \sum_{i=1}^m \phi(x_i) = \phi(y \vee x_m)+ \sum_{i=1}^m \phi(y_i) .$$ Clearly $y_i \le x_i$ for $1 \le i \le n$. The statement $P(1)$ follows from the modularity of $\phi$ after observing that $\phi(y) + \phi(x_1) = \phi(y \wedge x_1) + \phi(y \vee x_1) = \phi(y \vee x_1) + \phi(y_1)$ and $y_1 \le x_1 \le y \vee x_1$. Let us assume the truth of $P(k)$ for some positive integer $k \le n-1$. By the induction hypothesis and the modularity of $\phi$, we have 

\begin{align*}
\phi(y) + \sum_{i=1}^{k+1} \phi(x_i) &= \phi(x_{k+1}) + ( \phi(y) + \sum_{i=1}^{k} \phi(x_i)) = \phi(x_{k+1}) + (\phi(y \vee x_k) + \sum_{i=1}^{k} \phi(y_i)) \\
&=(\phi(x_{k+1}) + \phi(y \vee x_k)) +  \sum_{i=1}^{k} \phi(y_i)  \\
&= \phi(y \vee x_k \vee x_{k+1}) + \phi((y \vee x_k) \wedge x_{k+1})) + \sum_{i=1}^{k} \phi(y_i) \\
&= \phi(y \vee x_{k+1}) + \phi(y_{k+1}) + \sum_{i=1}^{k} \phi(y_i)  \\ 
&= \phi(y \vee x_{k+1}) + \sum_{i=1}^{k+1} \phi(y_i) 
\end{align*}

Further by the induction hypothesis, note that $y_1 \le \cdots \le y_k \le y \vee x_k$. As $y_k \le x_k \le x_{k+1}$, clearly $y_k \le  (y \vee x_k) \wedge x_{k+1} = y_{k+1} \le y \vee x_k \le y \vee x_{k+1}$. Thus $y_1 \le \cdots \le y_{k+1} \le y \vee x_{k+1}$. This finishes the proof of the assertion $P(k+1)$ and by induction, we have that $P(m)$ is true for $1 \le m \le n$. The assertion $P(n)$ is precisely the statement of the theorem.
\end{proof}

\begin{thm}
\label{thm:m_chain_rep}
\textsl{
Let $\mathcal{L}$ be a distributive lattice. For $n \in \mathbb{N}$ and a multiset $X := \{ x_1, x_2, \ldots, x_n \}$ with support contained in $\mathcal{L}$, we have $\sum_{i=1}^n \phi(x_i) = \sum_{i=1}^n \phi(X_{(i)})$.}
\end{thm}
\begin{proof}
We proceed inductively. The base case of $n=1$ is trivially true. For $n \in \mathbb{N}$, let us assume the truth of the assertion for any multisubset of cardinality $n$ with support contained in $\mathcal{L}$. Consider a multisubset $Y := \{ x_1, x_2, \ldots, x_{n+1} \}$ with support contained in $\mathcal{L}$ and let $X := \{ x_1, \ldots, x_n \}$. Recall that $$X_{(k)} = \wedge_{1 \le i_1 < \cdots < i_k \le n} (x_{i_1} \vee \cdots \vee x_{i_k}), 1\le k \le n.$$
As $\mathcal{L}$ is distributive, for $2 \le k \le n$, we have
$$x_{n+1} \vee \, X_{(k-1)} = \wedge_{1\le i_1 < \cdots < i_{k-1} \le n}(x_{n+1} \vee x_{i_1} \vee \cdots \vee x_{i_k}).$$ 
Thus for $2 \le k \le n$, we have $$(x_{n+1} \vee X_{(k-1)}) \wedge X_{(k)} = \wedge_{1\le i_1 < \cdots < i_k \le n+1}(x_{i_1} \vee \cdots \vee x_{i_k}) = Y_{(k)}.$$
It is straighforward to see that 
$$x_{n+1} \wedge X_{(1)} =  x_{n+1} \wedge (\wedge_{1\le i \le n}\, x_i) = Y_{(1)},$$
$$x_{n+1} \vee X_{(n)} =  x_{n+1} \vee (\vee_{1\le i \le n}\, x_i) = Y_{(n+1)}.$$

By the induction hypothesis, note that $\sum_{i=1}^n \phi(x_i) = \sum_{i=1}^n \phi(X_{(i)})$. Using Proposition \ref{prop:lat_ord}(ii) for the multichain $X_{(1)} \le \cdots \le X_{(n)}$ and $x_{n+1}$, we see that $\sum_{i=1}^{n+1} \phi(x_i) = \phi(x_{n+1}) + \sum_{i=1}^n \phi(x_i)  = \phi(x_{n+1}) + \sum_{i=1}^n \phi(X_{(i)}) =  \sum_{i=1}^{n+1} \phi(Y_{(i)})$. Thus by the principle of mathematical induction, the assertion of the theorem is true for all $n$ in $\mathbb{N}$.

\end{proof}

\subsection{Divisibility Lattice of GCD Domains}
\label{subsec:div_latt}

Let $\mathcal{R}$ be an integral domain. Consider the following equivalence relation on $\mathcal{R}$ : for $r, s $ in $\mathcal{R}$, define $r \sim_1 s$ if $r = us$ for an invertible element $u$ in $\mathcal{R}$ (in other words, if $r$ and $s$ are associates). We denote the equivalence class of $r$ by $(r)$. The set of equivalence classes of $\sim _1$ is denoted by $(\sR)$, and has a natural one-to-one correspondence with principal ideals of $\sR$. There is a natural partial order on $(\sR)$ based on divisibility, or equivalently, the inclusion of principal ideals. For $r_1, r_2$ in $\mathcal{R}$, we say that $(r_1) \le (r_2)$ if $r_2 \mid r_1$, or equivalently, $\sR r_1 \subseteq \sR r_2$. Further, the multiplicative structure on $\sR$ induces a monoidal structure on $(\sR)$ defined by $(r_1) \cdot (r_2) := (r_1 r_2)$ with identity $(1)$. Note that $\big( (\sR); \cdot \big)$ is a commutative monoid.

\begin{remark}
For integral domains $\sR _1, \sR _2$, let $\Phi : \sR _1 \to \sR _2$ be a ring homomorphism. As $\Phi$ takes invertible elements in $\sR _1$ to invertible elements in $\sR _2$, it induces an order preserving map from $\big( (\sR _1); \le \big)$ to $\big( (\sR _2); \le \big)$.
\end{remark}

\begin{definition}
A totally ordered multiset with support in $\big( (\sR); \le \big)$ is said to be a {\it divisor multichain}.
\end{definition}

Recall that every GCD domain is a LCM domain. Let $\sR$ be a GCD domain. For $r_1, r_2$ in $\sR$ we may define $(r_1) \wedge (r_2) := \textrm{LCM}\big( (r_1), (r_2) \big), (r_1) \vee (r_2) := \textrm{GCD} \big( (r_1), (r_2) \big)$ which makes $\big( (R); \vee, \wedge) \big)$ a lattice. The assertions in the following proposition are well-known (or straightforward) and we mention them without proof. 

\begin{prop}
\label{prop:prod_L1}
\textsl{
Let $\sR$ be a GCD domain. For $x, y , z$ in $(R)$, we have 
\begin{itemize}
\item[(i)] $xy = (x \wedge y)(x \vee y),$
\item[(ii)] $(0) \le x \le (1),$
\item[(iii)] $x \wedge (y \vee z) = (x \wedge y) \vee (x \wedge z),$
\item[(iv)] $x \vee (y \wedge z) = (x \vee y) \wedge (x \vee z).$
\end{itemize}
In other words, the identity map $\iota : \big( (\sR); \vee, \wedge \big) \rightarrow \big( (\sR); \cdot \big)$ is a modular map, and $\big( (\sR); \vee, \wedge \big)$ is a bounded, distributive lattice.
}
\end{prop}

\section{The Functor $\fU \fR$}
\label{sec:l_functor}

Before embarking on our study of ranked rings in \S \ref{sec:ranked_rings}, we first set up an algebraic framework for the discussion. In this section, we define a functor $\fU \fR$ from the category of rings to the category of commutative monoids. This not only serves to provide cleaner notation but also helps us formulate an algorithmic approach towards generating rank identities. Each element of the $\fU \fR$-monoid may be thought of as encoding a class of rank identities. This assertion is rigorously formulated in later sections (see Theorem \ref{thm:univ_rank_sys} and Proposition \ref{prop:rank_id_fund}).

Let $\mathcal{R}$ be a ring, and $n \in \N$. For matrices $X, Y$ in $M_n(\mathcal{R})$, we define a relation $X \sim_n  Y$ if there are invertible matrices $A, B \in GL_n(\mathcal{R} )$ such that $X = A\,YB$. Note that $\sim_n $ is an equivalence relation on $M_n(\mathcal{R})$ as shown below.
\begin{itemize}
\item[(i)] (Reflexivity) $X \sim_n X$, as $X = I_n X I_n$ for the identity matrix $I_n$ in $GL_n(\mathcal{R})$.
\item[(ii)] (Symmetry) For tuples $X, Y$ in $M_n(\mathcal{R})$, let $X \sim_n  Y$. For some elements $A, B \in GL_n(\mathcal{R})$, we have  $X = A\, Y B.$ Thus $Y = A^{-1}XB^{-1}$ which implies that $Y \sim_n X$.
\item[(iii)] (Transitivity) For tuples $X, Y, Z$ in $\mathcal{R}_{\langle n \rangle}$, let $X \sim_n  Y, Y \sim_n  Z $. Thus for some matrices $A, B, C, D \in GL_n(\mathcal{R})$, we have $X = A\,YB, Y = CZD$. As a result, $X = (AC)Z (D B)$ which implies that $X \sim_n Z$.
\end{itemize}

\begin{definition}
Let $\sR$ be a ring. On $M_{\mathrm{fin}}(\sR) := \sqcup_{i \in \N} M_i(\sR)$, we define a binary operation $\odot : M_{\mathrm{fin}}(\sR) \times M_{\mathrm{fin}}(\sR) \rightarrow M_{\mathrm{fin}}(\sR)$ which for matrices $X \in M_i(\mathcal{R}), Y \in M_{j}(\mathcal{R})$ gives the matrix $X \odot Y := \mathrm{diag}(X, Y)$ in $ M_{i+j}(\mathcal{R}) $.
\end{definition}

\begin{prop}
\label{prop:fund_L_func}
Let $\sR$ be a ring. For natural numbers $k, \ell, m$, let $X_1, Y_1 \in M_k(\mathcal{R}), X_2, Y_2 \in M_{\ell}(\mathcal{R})$ and $ X_3 \in M_{m}(\mathcal{R})$.
\begin{itemize}
\item[(i)] $X_1 \odot X_2 \sim_{k + \ell} X_2 \odot X_1$ in $M_{k+\ell}(\mathcal{R})$.
\item[(ii)] $X_1 \odot (X_2 \odot X_3) \sim_{k + {\ell} + m} (X_1 \odot X_2) \odot X_3$ in $M_{k+\ell+m}(\mathcal{R})$.
\item[(iii)] If $X_1 \sim_k Y_1$ in $M_k(\mathcal{R})$, $X_2 \sim_{\ell} Y_2$ in $M_{\ell}(\mathcal{R})$, then $X_1 \odot X_2 \sim_{k + \ell} Y_1 \odot Y_2$ in $M_{k + \ell}(\mathcal{R})$.
\end{itemize}
\end{prop}
\begin{proof}
(i) There is a permutation matrix $P \in GL_n(\mathcal{R})$ such that  $P (X_1 \odot X_2) P^{-1} = X_2 \odot X_1$. Thus $X_1 \odot X_2 \sim_{k+\ell}  X_2 \odot X_1$.
\vskip 0.05in

(ii) Straightforward from the definition of $\odot$.
\vskip 0.05in

(iii) Consider invertible elements $A_1, B_1 \in GL_m(\mathcal{R})$, $A_2, B_2 \in GL_n(\mathcal{R})$ such that $X_1 = A_1Y_1 B_1$ and $X_2 = A_2 Y_2 B_2$. The matrices $A_1 \odot A_2,  B_1 \odot B_2$ are in $GL_{m+n}(\sR)$ with inverses $A_1 ^{-1} \odot A_2 ^{-1}, B_1 ^{-1} \odot B_2 ^{-1} $, respectively. We have $X_1 \odot X_2 = (A_1 ^{-1} \odot A_2 ^{-1})\, (Y_1 \odot Y_2) (B_1 ^{-1} \odot B_2 ^{-1})$. Thus $X_1 \odot X_2 \sim_{m+n} Y_1 \odot Y_2$.

\end{proof}

\begin{definition}
For a $k \times k$ matrix $A \in M_{\mathrm{fin}}(\sR)$, define $\mathrm{size}(A) := k$. For $X, Y$ in $M_{\mathrm{fin}}(\sR)$, we define a relation $X \sim Y$ if there exists integers $m, n \ge 0$ such that $\mathrm{size}(X) + m = \mathrm{size}(Y)+n =: \ell$ and 
$X \odot \,{\bf 0}_m \sim_{\ell} Y \odot \,{\bf 0}_n$.
\end{definition}

\begin{lem}
\textsl{
For a ring $\sR$, the relation $\sim$ on $M_{\mathrm{fin}}(\sR)$ is an equivalence relation.
}
\end{lem}
\begin{proof}
It is immediate that $\sim$ is reflexive and symmetric. Let $X, Y, Z \in M_{\mathrm{fin}}(\sR)$ such that $X \sim Y$ and $Y \sim Z$. For some integers $m_1, n_1, m_2, n_2 \ge 0$, we have $\mathrm{size}(X) + m_1 = \mathrm{size}(Y) + n_1 =: \ell_1$ with $X \odot \,{\bf 0}_{m_1} \sim_{\ell_1} Y \odot \,{\bf 0}_{n_1}$, and $\mathrm{size}(Y) + m_2 = \mathrm{size}(Z) + n_2 =: \ell_2$ with $Y \odot \,{\bf 0}_{m_2} \sim_{\ell_2} Z \odot \,{\bf 0}_{n_2}$. If $n_1 \le m_2$, we see that $\mathrm{size}(X) + m_1 + m_2 - n_1 = \mathrm{size}(Y) + m_2 = \mathrm{size}(Z) + n_2 = \ell_2$ and $X \odot {\bf 0}_{m_1 + m_2 - n_1} \sim_{\ell_2} Y \odot {\bf 0}_{m_2} \sim_{\ell_2} Z \odot \,{\bf 0}_{n_2}$. If $n_1 > m_2$, we see that $\mathrm{size}(X) + m_1 = \mathrm{size}(Y) + m_2 + n_1 - m_2 = \mathrm{size}(Z) + n_2 + n_1 - m_2 = \ell _1$ and $X \odot {\bf 0}_{m_1} \sim_{\ell_2} Y \odot {\bf 0}_{n_1} \sim_{\ell_2} Z \odot \,{\bf 0}_{n_1 + n_2 - m_2}$. Thus $X \sim Z$.
 
Hence the relation $\sim$ is also transitive. In summary, the relation $\sim$ is an equivalence relation on $M_{\mathrm{fin}}(\sR)$.
\end{proof}

\begin{remark}
\label{rmrk:l_functor}
The equivalence class of a matrix $X \in 	M_{\mathrm{fin}}(\mathcal{R})$ under the equivalence relation $\sim$ is denoted by $[X]$. The set of equivalence classes of $	M_{\mathrm{fin}}(\mathcal{R})$ under $\sim$ is denoted by $\fU \fR (\mathcal{R}) := \big\{ [X]: X \in M_{\mathrm{fin}}(\mathcal{R}) \big\} $.  We reuse notation and define a binary operation $\odot : \fU \fR(\sR) \times \fU \fR(\sR) \rightarrow \fU \fR (\sR)$ given by $[X] \odot [Y] := [X \odot Y]$ which is well-defined by Proposition \ref{prop:fund_L_func}. Note that $[X] \odot [0] = [X \odot 0] = [X].$ Thus $\big( \fU \fR(\sR), \odot \big)$ possesses an identity $[0]$ and by Proposition \ref{prop:fund_L_func}, (i),(ii), we observe that $\fU \fR (\mathcal{R})$ is a commutative monoid. 
\end{remark}

\begin{prop}[Functoriality of $\fU \fR$]
\label{prop:L_functor}
\textsl{
For rings $\sR _1, \sR _2$, let $\Phi : \sR _1 \to \sR _2$ be a ring homomorphism. By entrywise application of $\Phi$ on matrices over $\sR _1$, we obtain ring homomorphisms from $M_n(\sR _1)$ to $M_n(\sR _2)$ for all $n \in \mathbb{N}$, which by abuse of notation, we again denote by $\Phi$. Then the map $\fU \fR (\Phi) : \fU \fR(\sR _1) \rightarrow \fU \fR(\sR _2)$ given by $\fU \fR (\Phi)([X]) = [\Phi(X)]$ (for $X \in M_{\mathrm{fin}}(\sR _1)$) is well-defined and has the following properties : 
\begin{itemize}
\item[(i)] $\fU \fR (\Phi)([0]) = [0]$.
\item[(ii)] For $[X], [Y] \in \fU \fR(\sR _1)$, we have $\fU \fR (\Phi)([X] \odot [Y]) = \fU \fR(\Phi)([X]) \odot \fU \fR (\Phi)([Y])$.
\end{itemize} 
In other words, $\fU \fR (\Phi)$ is a homomorphism between the commutative monoids $\fU \fR (\sR _1)$ and $\fU \fR (\sR _2)$.}
\end{prop}

\begin{proof}
For $n \in \mathbb{N}$, note that $\Phi$ induces a ring homomorphism from $M_n(\sR _1)$ to $M_n(\sR _2)$ by applying it entrywise. Thus for all $n \in \N$,  $\Phi({\bf 0}_n) = {\bf 0}_n$ and $\Phi(GL_n(\sR _1)) \subseteq GL_n(\sR _2)$. As a consequence, the map $\fU \fR (\Phi) : \fU \fR(\sR _1) \rightarrow \fU \fR (\sR _2)$ given by $\fU \fR (\Phi)([X])  = [\Phi(X)]$ is well-defined.

(i) As $\Phi(0) = 0$, clearly $\fU \fR (\Phi)([0]) = [0]$.

(ii) For $[X], [Y] \in \fU \fR(\mathcal{R})$, we have $\fU \fR (\Phi)([X] \odot [Y]) = \fU \fR(\Phi)([X \odot Y]) = [\Phi(X \odot Y)] = [\Phi(X) \odot \Phi(Y)] = [\Phi(X)] \odot [\Phi(Y)] = \fU \fR(\Phi)([X]) \odot \fU \fR (\Phi)([Y])$. 
\end{proof}

\begin{notation}
Let $\sR$ be a ring. For a subset $\mathscr{S}$ of $M_{\mathrm{fin}}(\sR)$, we define $[\mathscr{S}] := \{ [X] : X \in \sS \} \subseteq \fU \fR (\sR)$. In this notation, $\fU \fR (\sR) = [M_{\mathrm{fin}}(\sR)]$. 
\end{notation}

\begin{definition}
\label{def:l_functor}
We denote the category of (unital) rings by {\it Rings} and the category of commutative monoids by {\it CMon}. The mapping from Rings to CMon which maps a ring $\sR$ to the
commutative monoid $\fU \fR (\sR)$, and a ring homomorphism $\Phi : \sR _1 \to \sR _2$ between rings $\sR _1, \sR _2$ to the monoid homomorphism $\fU \fR (\Phi) : \fU \fR (\sR _1) \to \fU \fR (\sR _2)$, is a covariant functor. Furthermore, $\fU \fR (\sR)$ is said to be the $\fU \fR$-monoid of $\sR$.
\end{definition}

\section{Ranked rings}
\label{sec:ranked_rings}

In this section, we define the notion of {\it ranked rings} (rings with a rank-system) and study some examples pertinent to operator theory. For example, we equip every Murray-von Neumann algebra with a center-valued rank-system which turns out to be the universal rank-system in an appropriate sense. We also show that every von Neumann rank-ring may be viewed as a ranked ring in a natural manner.

\begin{definition}
\label{def:ranked}
Let $\sR$ be a ring and $\sM$ be a commutative monoid. An {\it $\sM$-valued rank-system} or simply {\it rank-system} for $\sR$ is a map $\rho : M_{\mathrm{fin}}(\sR) \rightarrow \sM$ satisfying the following properties : 
\begin{itemize}
\item[(a)] For $X, Y \in M_{\mathrm{fin}}(\sR)$, $\rho(X \odot Y) = \rho(X) + \rho(Y)$ ,
\item[(b)] For $n \in \N$ and $X, Y \in M_n(\sR)$ with $Y \in GL_n(\sR)$, we have $\rho(XY) = \rho(YX) = \rho(X)$.
\item[(c)] For $n \in \N$, $\rho({\bf 0}_n) = 0$.
 \end{itemize}
With the $\sM$-valued rank-system $\rho$, $\sR$ is said to be {\it $(\sM, \rho)$-ranked} or simply, {\it $\rho$-ranked}. A rank-system $\rho$ is said to be {\it non-degenerate} if $\rho(X) = 0$ for $X \in M_n(\sR)$ implies that $X = {\bf 0}_n$.  It may be helpful for the reader to mentally substitute any occurrence of $\rho$ in this article with the word `rank'.
\end{definition}

\begin{thm}[The universal rank-system]
\label{thm:univ_rank_sys}
\textsl{
Let $\sR$ be a ring and $\sM$ be a commutative monoid. The map $\mathrm{rank}_{u} : M_{\mathrm{fin}}(\sR) \to \big( \fU \fR (\sR); \odot \big)$ defined by $\mathrm{rank}_{u}(X) = [X]$ is a surjective rank-system on $\sR$. Furthermore, if $\rho : M_{\mathrm{fin}}(\sR) \to \sM$ is a rank-system on $\sR$, then there is a unique monoid homomorphism $\varphi : \fU \fR (\sR) \to \sM$ such that $\rho = \varphi \circ \mathrm{rank}_{u}$. In other words, every rank-system on $\sR$ factors through the rank-system, $\mathrm{rank}_{u}$.
}
\end{thm}
\begin{proof}
It is straightforward from the definition of the $\fU \fR$-monoid that $\mathrm{rank}_u$ is a surjective rank-system on $\sR$.

For $X, Y \in M_{\mathrm{fin}}(\sR)$, let $\mathrm{rank}_u(X) = \mathrm{rank}_u(Y)$. Then $X \odot {\bf 0}_k \sim_m Y \odot {\bf 0}_{\ell}$ for some $k, \ell, m \in \N$. From property (b) in Definition \ref{def:ranked}, we have $\rho(X \odot {\bf 0}_k) = \rho(Y \odot {\bf 0}_{\ell})$. Using property (a), (c), we have $\rho(X \odot {\bf 0}_k) = \rho(X) + \rho({\bf 0}_k) = \rho(X)$, and $\rho(Y \odot {\bf 0}_k) = \rho(Y) + \rho({\bf 0}_{\ell}) = \rho(Y)$. Thus $\rho(X) = \rho(Y)$. Since the mapping $\mathrm{rank}_u$ is surjective, we observe that the mapping $\varphi : \fU \fR(\sR) \to G$ given by $\varphi \big( \mathrm{rank}_u(X) \big) = \rho(X)$ is well-defined.

Note that $\varphi([0]) = \rho(0) = 0$. Let $X, Y \in M_{\mathrm{fin}}(\sR)$. Using property (a) in Definition \ref{def:ranked}, we have $\varphi \big( \mathrm{rank}_u(X) \odot \mathrm{rank}_u(Y) \big) = \varphi \big( \mathrm{rank}_u(X \odot Y) \big) = \rho(X \odot Y) = \rho(X) + \rho(Y)$. Thus $\varphi$ is a monoid homomorphism.
\end{proof}

\begin{remark}
In light of Theorem \ref{thm:univ_rank_sys}, we may interprete the moniker $\fU \fR$ used in \S \ref{sec:l_functor} as an abbreviation for ``universal rank-system".
\end{remark}

\begin{remark}
Let $\sR _1$ be a ring, and $\sR _2$ be a $(\sM, \rho)$-ranked ring. Let $\Phi : \sR_1 \to \sR_2$ be a ring homomorphism. Then the mapping $\Phi^*(\rho) : M_{\mathrm{fin}}(\sR_1) \to \sM$ defined by $\Phi ^*(\rho)(X) = \rho(\Phi(X))$ is a $\sM$-valued rank-system on $\sR_1$, called the {\it pullback} of $\rho$ under $\Phi$.
\end{remark}

\begin{cor}
\label{cor:UR_mon_morita}
\textsl{
Let $\sR$ be a ring. Then $\fU \fR (\sR) \cong \fU \fR \big( M_n(\sR) \big) $ for all $n \in \N$.
}
\end{cor}
\begin{proof}
Let $\mathrm{rank}_{u,1}$ denote the universal rank-system on $\sR$, and $\mathrm{rank}_{u,n}$ denote the universal rank-system on $M_n(\sR)$. Consider the natural injection $\iota : M_{\mathrm{fin}}\big( M_n(\sR) \big) \hookrightarrow M_{\mathrm{fin}}(\sR)$. It is easy to see that the restriction of $\mathrm{rank}_{u,1}$ to $M_{\mathrm{fin}} \big( M_n(\sR) \big)$ defines a rank-system for $M_n(\sR)$. By virtue of Theorem \ref{thm:univ_rank_sys}, there is a unique monoid homomorphism $\iota ^* : \fU \fR \big( M_n(\sR) \big) \to \fU \fR (\sR)$  such that the following diagram commutes.

\[\begin{tikzcd}
M_{\mathrm{fin}} \big( M_n(\sR) \big) \arrow{r}{\iota} \arrow[swap]{d}{\mathrm{rank}_{u,n}} & M_{\mathrm{fin}}(\sR) \arrow{d}{\mathrm{rank}_{u,1}} \\
\fU \fR \big( M_n(\sR) \big) \arrow{r}{\iota ^*} & \fU \fR (\sR)
\end{tikzcd}
\]

Let $X, Y \in M_{\mathrm{fin}}(M_n(\sR))$ such that $\mathrm{rank}_{u, 1}(X) = \mathrm{rank}_{u,1}(Y)$, that is, $X \odot {\bf 0}_k \sim_m Y \odot {\bf 0}_{\ell}$ for some $k, \ell, m \in \N$. Let $m' \in \N$ such that $m+m'$ is divisible by $n$. Then by Proposition \ref{prop:fund_L_func}, we have $X \odot {\bf 0}_{k+m'} \sim_{m+m'} Y \odot {\bf 0}_{\ell + m'}$ and thus $\mathrm{rank}_{u, n}(X) = \mathrm{rank}_{u, n}(Y)$. This shows that $\iota ^*$ is injective.

Let $m, m' \in \N$ such that $m + m'$ is divisible by $n$. For $X \in M_m(\sR)$, note that $X \odot {\bf 0}_{m'} \in M_{\mathrm{fin}}(M_n(\sR))$ so that $\iota ^* (\mathrm{rank}_{u,n}(X \odot {\bf 0}_{m'})) = \mathrm{rank}_{u,1}(X)$. Thus $\iota ^*$ is surjective.

In summary, $\iota ^*$ defines a monoid isomorphism between $\fU \fR \big( M_n(\sR) \big)$ and $\fU \fR (\sR)$.
\end{proof}

\begin{example}
\label{ex:ranked_ring}
Let $\K$ be a field and $\mathrm{rank} : M_{\mathrm{fin}}(\K) \to \Z ^+$ denote the usual rank function (column-rank or equivalently, row-rank). This defines a $\mathbb{Z}^{+}$-valued rank-system for $\K$. In \S \ref{subsec:EDD_mon}, we show that this is, in fact, the universal rank-system for $\K$.
\end{example}

\subsection{The center-valued rank-system on finite von Neumann algebras and Murray-von Neumann algebras}

\label{subsec:center_valued}

Let $\mathscr{R}$ be a finite von Neumann algebra acting on a complex Hilbert space $\mathscr{H}$, and let $\mathscr{C}$ denote the center of $\mathscr{R}$. The canonical center-valued trace from $\mathscr{R}$ to $\mathscr{C}$ is denoted by $\tau : \mathscr{R} \rightarrow \mathscr{C}$. For $n \in \mathbb{N}$,  $M_n(\mathscr{R})$ is a finite von Neumann algebra acting on $\oplus_{i=1}^n \mathscr{H}$. There is a natural isomorphism between the center of $M_n(\mathscr{R})$ and $\mathscr{C}$. Let $\tau_n$ be the center-valued trace for $M_n(\mathscr{R})$ taking values in $\mathscr{C}$ via the isomorphism. For an operator $\mathbf{A} = (A_{ij})_{1 \le i, j \le n}$ in $M_n(\mathscr{R})$, it is not hard to see that $\tau_n(\mathbf{A}) = \frac{1}{n} \big( \tau(A_{11}) + \cdots + \tau(A_{nn}) \big)$. Note that $\tau_1 \equiv \tau$. 

\begin{definition}
The {\it rank} of an operator $A$ in $\afr$ is defined as $\fr _c (A) := \tau \big( R(A) \big)$.
\end{definition}

In the proposition below, we prove some fundamentally useful properties of the $\mathscr{C}^{+}$-valued rank, some of which may be found in \cite{tr_rank}.
\begin{prop}
\label{prop:rank_prop}
\textsl{
Let $A, B, E, F$ be operators in $\afr$ with $E, F$ being projections. Then we have the following:
\begin{itemize}
\item[(i)] $E \sim F$ if and only if $\fr _c(E) = \fr _c(F)$.
\item[(ii)] $E \preccurlyeq F$ if and only if $\fr _c(E) \le  \fr _c(F)$.
\item[(iii)] $\fr _c(R(A)) = \fr _c(A)$.
\item[(iv)] if $0 \le A \le B$, then $\fr _c(A) \le \fr _c(B)$ with equality if and only if $R(A) = R(B)$.
\item[(v)] $\fr _c(A) = \fr _c(A^*) = \fr _c(A^* \afrprod A) = \fr _c(A \afrprod A^*)$.
\item[(vi)] $\fr _c(AB) \le \min \big\{ \fr _c(A), \fr _c(B) \big\} $.
\item[(vii)] If $B$ is invertible, $\fr _c(A \afrprod B) = \fr _c(B \afrprod A)=\fr _c(A)$.
\item[(viii)] $\fr _c(A) = 0 \Longleftrightarrow A = 0$.
\end{itemize}
} 
\end{prop}
\begin{proof}
\noindent (i) Note that $E \sim F$ if and only if $\tau(E) = \tau(F)$. For projections $E, F$ in $\mathscr{R}$, it is straightforward to see that $R(E) = E, R(F) = F$ and thus $\fr _c(E) = \tau(E), \fr _c(F) = \tau(F)$.
\vskip 0.03in

\noindent (ii) Follows from part (i) and the fact that $E \preccurlyeq F$ if and only if $\tau(E) \le \tau(F)$.
\vskip 0.03in

\noindent (iii) The assertion holds as $R \big( R(A) \big) = R(A)$.
\vskip 0.03in

\noindent (iv) If $0 \le A \le B$, then $N(B) \le N(A)$. Thus by Lemma \ref{lem:range_proj_rel} (i), $R(A)  = I- N(A) \le  I - N(B) = R(B)$. The conclusion follows from part (i), (ii).
\vskip 0.03in

\noindent (v) Follows from part (i) of the proposition and Lemma \ref{lem:range_proj_rel}, (iii).
\vskip 0.03in

\noindent (vi) Note that since $\mathscr{C}$ is a commutative $C^*$-algebra, it makes sense to talk about the minimum of two self-adjoint operators in $\mathscr{C}$. By Lemma \ref{lem:range_proj_rel}, (iv), $R(A \afrprod B) = R \big( A \afrprod R(B) \big) \le R(A)$. Thus $\fr(AB) \le \fr(A)$. Similarly, using part (v), we conclude that $\fr(A \afrprod B) = \fr(B^* \afrprod A^*) \le \fr(B^*) = \fr(B)$.
\vskip 0.03in

\noindent (vii) If $B$ is invertible, note that $R(B) = R(B^*) = I$. Using Lemma \ref{lem:range_proj_rel}, (iv), we see that $\fr _c(A \afrprod B) = \fr _c \big( A \afrprod  R(B) \big) = \fr _c(A).$ Similarly, using part (v), we arrive at $\fr _c(B \afrprod A) = \fr _c(A^* \afrprod  B^*) = \fr _c \big( A^* \afrprod R(B^*) \big) = \fr _c(A^*) = \fr _c(A)$.
\vskip 0.03in

\noindent (viii) Follows from the faithfulness of $\tau$ and because $R(A) = 0$ if and only if $A = 0$.
\end{proof}

\begin{definition}
We define a map $\rho _c : \sqcup_{n \in \N} M_n(\mathscr{R})_{\textrm{aff}} \rightarrow \mathscr{C}^{+}$ by $\rho _c(A) = n \tau_n \big( R(A) \big)$ for $A$ in $M_n(\mathscr{R})_{\textrm{aff}}$. Note that $\fr _c$ is the restriction of $\rho _c$ to $\afr$.
\end{definition}

\begin{cor}
\label{cor:vna_rank_sys}
\textsl{
For operators $A, B$ in $M_n(\mathscr{R})_{\textrm{aff}}$, the following holds.
\begin{itemize}
\item[(i)] $\rho _c(A) = \rho _c(A^*) = \rho _c(A^* \afrprod A) = \rho _c(A \afrprod A^*)$.
\item[(ii)] If $B$ is invertible, $\rho _c(A \afrprod B) = \rho _c(B \afrprod A) = \rho _c(A)$.
\item[(iii)] $\rho _c(A) = 0 \Longleftrightarrow A = 0$.
\end{itemize}
}
\end{cor}

\begin{proof}
The assertions follow from Proposition \ref{prop:rank_prop} in the context of the finite von Neumann algebra $M_n(\mathscr{R})$.
\end{proof}
 
\begin{thm}
\label{thm:vna_rank_system}
\textsl{
The mapping $\rho _c : M_{\mathrm{fin}}(\afr) \to \mathscr{C}^{+}$ (using the isomorphism $M_n(\afr) \cong M_n(\mathscr{R})_{\textrm{aff}}$) defines a $\mathscr{C}^{+}$-valued rank-system on $\afr$.
}
\end{thm}
\begin{proof}
Note that Corollary \ref{cor:vna_rank_sys} has the ingredients to verify property \ref{def:ranked}(b), \ref{def:ranked}(c) for the rank-system $\rho _c$. We prove property \ref{def:ranked}(a) below.  It is clear from their respective definitions that for operators $X \in M_k(\mathscr{R}), Y \in M_{\ell}(\mathscr{R}) , k , \ell \in \mathbb{N}$ we have $(k+\ell)\tau_{k+\ell}(X \odot Y) = k \tau_k(X) + \ell \tau_{\ell}(Y)$.\ For $A \in M_k(\afr), B \in M_{\ell}(\afr)$, since $R(A \odot B) = R(A) \odot R(B)$ in $M_{k + \ell}(\mathscr{R})$, we have $\rho_c(A \odot B) = \rho_c(A) + \rho_c(B).$ From Remark \ref{rmrk:mvna_iso}, we conclude that $\rho _c$ equips $\afr$ with a $\mathscr{C}^+$-valued rank-system.
\end{proof}

In \S \ref{subsec:UR_mon_MvN}, we show that $\rho _c$ is, in fact, the universal rank-system for $\afr$.

\begin{remark}
\label{rmrk:fvn_rank_system}
The pullback of the rank-system $\rho$ on $\afr$ under the inclusion $\mathscr{R} \hookrightarrow \afr$ begets a $\mathscr{C}^{+}$-valued rank-system for $\mathscr{R}$. With this rank-system, we view $\mathscr{R}$ as a ranked $\C$-algebra or a ranked complex Banach algebra.
\end{remark}

\subsection{$\R ^{+}$-valued rank-systems on von Neuman rank-rings}
\label{subsec:rank_vNreg}

\begin{lem}
\label{lem:rank_fun_inv}
\textsl{
Let $\sR$ be a von Neumann regular ring and $\fr $ be a rank-function on $\sR$. For $x, y \in \sR$ with $y$ invertible, we have $\fr (x) = \fr (xy) = \fr (yx)$.
}
\end{lem}
\begin{proof}
Note that $\fr (x) = \fr \big( (xy)y^{-1} \big) = \fr \big( y^{-1}(yx) \big) \le \min \{ \fr (xy), \fr (yx) \} \le \fr (x).$ Thus $\fr (x) = \fr (xy) = \fr (yx)$.
\end{proof}

\begin{remark}
In the notation of Theorem \ref{thm:pushout_rank}, we observe that $\delta _n ^*$ defines a one-to-one correspondence between the set of rank-functions on $M_n(\sR)$ normalized to $n$ and the set of rank-functions on $\sR$ normalized to $1$.
\end{remark}

\begin{lem}
\label{lem:rank_fun_rank_sys}
\textsl{
Let $\sR$ be a von Neumann regular ring and $\fr ^{(1)}$ be a rank-function on $\sR$. For $k \in \N$, let $\fr ^{(k)} := (\delta _k ^*)^{-1}(\fr ^{(1)})$. Let $n$ be a positive integer. 
\begin{itemize}
\item[(i)] For an idempotent $e \in \sR$ and $1 \le m \le n$, we have $\fr ^{(n)}({\bf 0}_{m-1} \odot e \odot {\bf 0}_{n-m}) = \fr ^{(1)}(e)$.
\item[(ii)] For idempotents $e_1, \ldots, e_n \in \sR$, we have $\fr ^{(n)}(\odot_{i=1}^n e_i) = \sum_{i=1}^n \fr ^{(1)}(e_i)$;
\item[(iii)] For $1 \le m \le n$ and $X \in M_m(\sR)$, we have $\fr ^{(m)}(X) = \fr ^{(n)}(X \odot {\bf 0}_{n-m}) = \fr ^{(n)}({\bf 0}_{n-m} \odot X)$.
\item[(iv)] For $1 \le m \le n$ and $X_1 \in M_m(\sR), X_2 \in M_{n-m}(\sR)$, we have $\fr ^{(n)}(X_1 \odot X_2) = \fr ^{(m)}(X_1) + \fr ^{(n-m)}(X_2).$
\end{itemize}
}
\end{lem}
\begin{proof}
\noindent (i) Let $E_m := {\bf 0}_{m-1} \odot e \odot {\bf 0}_{n-m} \in M_n(\sR)$ for $1 \le m \le n$. Note that the $E_m$'s are mutually orthogonal similar idempotents in $M_n(\sR)$ with the similarity implemented by appropriate permutation matrices, and $\sum_{i=1}^n E_i = \odot_{i=1}^n e$. By Lemma \ref{lem:rank_fun_inv}, $\fr ^{(n)}(E_i) = \fr ^{(n)}(E_1)$ for $1 \le i \le n$. By property (iii), $n \fr ^{(1)}(e) =  \fr ^{(n)} \big( \delta_n(e) \big) = \fr ^{(n)}(\odot_{i=1}^n e) =  \sum_{i=1}^n \fr ^{(n)}(E_i) = n \fr ^{(n)}(E_1)$. Thus $\fr ^{(1)}(e)  = \fr ^{(n)}(E_m) = \fr ^{(n)}({\bf 0}_{k-1} \odot e \odot {\bf 0}_{n-m})$ for $1 \le m \le n$.
\vskip 0.1in

\noindent (ii) Define $F_m := {\bf 0}_{m-1} \odot e_m \odot {\bf 0}_{n-m}$ for $1 \le m \le n$. Note that the $F_m$'s are mutually orthogonal idempotents and from part (i), we have $\fr ^{(n)}(F_m) = \fr ^{(1)}(e_m)$. Thus $\fr ^{(n)}(\odot_{i=1}^n e_i) = \fr ^{(n)}(\sum_{i=1}^n E_i) = \sum_{i=1}^n \fr ^{(n)} (E_i) = \sum_{i=1}^n \fr ^{(1)}(e_i).$
\vskip 0.1in

\noindent (iii) Let $\rho : M_m(\sR) \to \R ^{+}$ be defined by $\rho(X) = \fr ^{(n)}(X \odot {\bf 0}_{n})$. It is easy to see that $\rho$ is a rank-function on $M_m(\sR)$. Since $\sR$ is von Neumann regular, we may choose $y \in \sR$ be such that $xyx = x$. Note that $e = xy$ is an idempotent and $\delta_m(x) = \delta_m(x) \delta_m(y) \delta_m(x)$ with $\delta_m(e) = \delta_m(x) \delta_m(y)$. Thus for any rank-function on $M_m(\sR)$, the rank of $\delta_m(x)$ is equal to the rank of $\delta_m(e)$; in particular, this holds for $\rho$ and $\fr ^{(m)}$. Using part (ii),  $\rho \big( \delta_m(x) \big) = \rho(\delta_m(e)) = m \fr ^{(1)}(e) = \fr ^{(m)} \big(\delta_m(e) \big) = \fr ^{(m)} \big( \delta_m(x) \big)$. From Theorem \ref{thm:pushout_rank}, we conclude that $\rho = \fr ^{(m)}$. The second part of the equality can be shown similarly.
\vskip 0.1in

\noindent (iv) Let $Y_1 \in M_m(\sR), Y_2 \in M_{n-m}(\sR)$ such that $X_1 Y_1 X_1 = X_1, X_2 Y_2 X_2 = Y_2$. Let $E_1 := X_1 Y_1 \in M_m(\sR)$ and $E_2 := X_2 Y_2 \in M_{n-m}(\sR)$ be idempotents. Since $E_1 \odot {\bf 0}_{n-m}$ and ${\bf 0}_m \odot E_2$ are mutually orthogonal, we have $\fr ^{(n)}(X \odot Y) = \fr ^{(n)}(E_1 \odot E_2) = \fr ^{(n)}(E_1 \odot {\bf 0}_{n-m}) + \fr ^{(n)}({\bf 0}_{m} \odot E_2).$ From part (iii), we conclude that $\fr ^{(n)}(X \odot Y) = \fr ^{(m)}(X_1) + \fr ^{(n-m)}(X_2)$.
\end{proof}

\begin{thm}
\label{thm:rank_sys_ring}
\textsl{
Let $\sR$ be a von Neumann regular ring and $\fr ^{(1)}$ be a rank-function on $\sR$. For $k \in \N$, let $\fr ^{(k)} := (\delta _k ^*)^{-1}(\fr ^{(1)})$. Then the mapping $\rho : M_{\mathrm{fin}}(\sR) \to \R ^{+}$ given by $\rho (A) := \fr ^{(k)}(A)$ for $A \in M_n(\sR)$, defines a $\R^{+}$-valued rank-system on $\sR$.
}
\end{thm}
\begin{proof}
This is immediate from Lemma \ref{lem:rank_fun_inv} and Lemma \ref{lem:rank_fun_rank_sys}, (iv).
\end{proof}

In Theorem \ref{thm:rank_sys_ring_conv}, it is shown that every $\R ^{+}$-valued rank-system on a von Neumann regular ring restricts to a rank-function on the ring.

\section{The $\fU \fR$-monoid of some rings of interest}
\label{sec:L_bezout}

In this section, we compute the $\fU \fR$-monoid of elementary divisor domains, and Murray-von Neumann algebras. Pertinent examples of elementary divisor domains include fields, univariate polynomial rings over fields, the ring of entire functions, and abelian Murray-von Neumann algebras.

\subsection{The $\fU \fR$-monoid of elementary divisor domains}
\label{subsec:EDD_mon}
Let $\sR$ be an elementary divisor domain. Let $D_{\mathrm{fin}}(\sR)$ be the set of diagonal matrices in $M_{\mathrm{fin}}(\sR)$. Every matrix in $M_n(\sR)$ is $\sim_n$-equivalent to a diagonal matrix. Thus $\fU \fR (\sR) = [D_{\mathrm{fin}}(\sR)]$ and $\fU \fR (\sR)$ is generated by $[\sR] = [M_1(\sR)]$. Note that every elementary divisor domain is a B\'e{}zout domain, and hence a GCD domain.

\begin{lem}
\label{lem:x_1-x}
\textsl{
Let $\sR$ be a commutative ring and $a, b \in \sR$ such that $\sR a + \sR b = \sR$. Then for all $c \in \sR$, we have $$\mathrm{rank}_u(ca) \odot \mathrm{rank}_u(cb) = \mathrm{rank}_u(cab) \odot \mathrm{rank}_u(c).$$
}
\end{lem}
\begin{proof}
There are elements $r, s$ in $\mathcal{R}$ be such that $ra+sb = 1$. Using Lemma \ref{lem:invert_matrix}, we note that the matrices
$
\begin{bmatrix}
sb & -a \\
r & 1\\
\end{bmatrix},
\begin{bmatrix}
b & 1\\
-ra & s\\
\end{bmatrix} 
$
are invertible as 
$\mathrm{det} \Big(
\begin{bmatrix}
sb & -a \\
r & 1\\
\end{bmatrix} \Big) = 1,
\mathrm{det} \Big(
\begin{bmatrix}
b & 1\\
-ra & s\\
\end{bmatrix} \Big) = 1
$. From the matrix computation below,
$$
\begin{bmatrix}
sb & -a \\
r & 1
\end{bmatrix}
\begin{bmatrix}
ca & 0 \\
0 & cb
\end{bmatrix}
\begin{bmatrix}
b & 1\\
-ra & s
\end{bmatrix} = 
\begin{bmatrix}
c a b  & 0\\
0 & c
\end{bmatrix},
$$
we conclude that $\mathrm{rank}_u(ca) \odot \mathrm{rank}_u(cb) = \mathrm{rank}_u(cab) \odot \mathrm{rank}_u(c)$.
\end{proof}

\begin{lem}
\label{lem:ufd_fund}
\textsl{
Let $\mathcal{R}$ be a B\'{e}zout domain. The mapping from $ \big( (\sR); \vee, \wedge \big)$ to $\big( \fU \fR(\sR); \odot \big)$ given by, $(x) \mapsto [x]$, is well-defined and an injective modular map.
}
\end{lem}
\begin{proof}
Let us denote the mapping by $\Phi$. Note that $(x) = (y)$ if and only if $x \sim_1 y$. Thus if $(x) = (y)$, we have $[x]=[y]$ which shows that $\Phi$ is well-defined.

Let $x, y \in \sR$ such that $[x] = [y]$. For some $k \in \N$, we have $x \odot \,{\bf 0}_{k-1} \sim_{k} y \odot \,{\bf 0}_{k-1}$. If $P, Q \in GL_k(\sR)$ such that $P(x \odot \,{\bf 0}_{k-1} ) Q = y \odot \,{\bf 0}_{k-1}$, then $P_{11}Q_{11} x = y$. Thus $\sR y \subseteq \sR x$. Similarly $\sR x \subseteq \sR y$ leading us to the conclusion that $(x) = (y)$. Thus $\Phi$ is injective.

Let $x, y \in \sR$. Since $\sR$ has the B\'{e}zout property, there exists $z \in \sR$ such that $\sR x + \sR y = \sR z$, and $(x) \vee (y) = (z)$. There are $x', y' \in \sR$ such that $x = z x', y = zy'$ and $\sR x'  + \sR y' = \sR$. Furthermore, $(x) \wedge (y) = (zx'y')$. By Lemma \ref{lem:x_1-x}, we have $$\Phi \big( (x) \big) \odot \Phi \big( (y) \big) = \Phi \big( (x) \wedge (y) \big) \odot \Phi \big( (x) \vee (y) \big).$$
Thus $\Phi$ is a modular map.
\end{proof}

For a B\'{e}zout domain $\sR$, by virtue of Lemma \ref{lem:ufd_fund}, we may identify $(\sR)$ with $[\sR]$. For $x \in \sR$, we use $[x]$ and $(x)$ interchangeably. The $\odot$-sum of two elements in $(\sR)$ has the obvious meaning derived from $[\sR]$, and lies in $\fU \fR(\sR)$.

\begin{thm}
\label{thm:fund_ufd_algo}
\textsl{
Let $\sR$ be a B\'{e}zout domain. Let $x_1 \le \cdots \le x_n$ be a multichain in $(\sR)$, and $y$ be an element of $(\sR)$. Let $y_1 = y \wedge x_1$, $y_i = (y \vee x_{i-1}) \wedge x_i$ for $2 \le i \le n$, and $y_{n+1} = y \vee x_n$. Then we have
\begin{itemize}
\item[(i)] $y_1 \le y_2 \le \cdots \le y_{n+1} $,
\item[(ii)] $y \odot (\odot_{i=1}^n x_i) = \odot_{i=1}^{n+1} y_i$.
\end{itemize} 
}
\end{thm}
\begin{proof}
Follows from Lemma \ref{lem:ufd_fund} and Proposition \ref{prop:lat_ord}.
\end{proof}

\begin{thm}
\label{thm:div_chain_rep}
\textsl{
Let $\sR$ be a B\'{e}zout domain. For $n \in \mathbb{N}$ and a multisubset $X := \{ x_1, x_2, \ldots, x_n \}$ with support contained in $(\sR)$, we have $\odot_{i=1}^n x_i = \odot_{i=1}^n X_{(i)}$.}
\end{thm}
\begin{proof}
Follows from Lemma \ref{lem:ufd_fund} and Theorem \ref{thm:m_chain_rep}.
\end{proof}

\begin{prop}
\label{prop:div_chain_unique}
\textsl{
Let $\sR$ be a B\'{e}zout domain. For $n \in \mathbb{N}$ and multichains $x_1 \le \cdots \le x_n, y_1 \le \cdots \le y_n$ in $(\sR)$, we have $\odot_{i=1}^n x_i = \odot_{i=1}^n y_i$ if and only if $x_i = y_i$ for $1 \le i \le n$.
}
\end{prop}
\begin{proof}
We follow the standard strategy that is used to prove the uniqueness of the Smith normal form for matrices over elementary divisor domains.

Let $x_1 = [r_1], \ldots, x_n = [r_n]$ and $y_1 = [s_1], \ldots, y_n = [s_n]$ for elements $r_1, \ldots, r_n,$ $ s_1, \ldots, s_n \in \mathcal{R}$. Let $X := \oplus_{i=1}^n r_i, Y := \oplus_{i=1}^n s_i$ be diagonal matrices in $M_n(\mathcal{R})$. Let $d_k(X)$ denote the GCD of the $(n-k+1) \times (n-k+1)$ minors of $X$. Define $d_k(Y)$ similarly. Clearly $d_k(X) = \prod_{i=n-k+1}^n x_i, d_k(Y) = \prod_{i=n-k+1}^n y_i$ for $1 \le k \le n$.

Since $PXQ = Y$ for some invertible matrices $P, Q \in GL_n(\sR)$, for $1 \le k \le n$, we have $d_k(Y) \le d_k(X)$ in $(\sR)$, and similarly $d_k(X) \le d_k(Y)$ in $(\sR)$. Thus $d_k(X) = d_k(Y)$ for $1 \le k \le n$. Note that $y_{k+1} d_k(Y) = x_{k+1} d_k(X)$ for $1 \le k \le n-1$. Moreover, $d_k(X) = [0]$ ($d_k(Y) = [0]$, respectively) if and only if $x_k = [0]$ ($y_k = [0]$, respectively). We conclude that $x_k = y_k$ for $1 \le k \le n$.
\end{proof}

\begin{cor}
\label{cor:odot_equal_prod}
\textsl{
Let $\sR$ be a B\'{e}zout domain. For $n \in \mathbb{N}$ and multisubsets $X := \{ x_1, \ldots, x_n \},$ $Y := \{ y_1, \ldots, y_n \}$ with their respective supports contained in $(\sR)$, we have $\odot_{i=1}^n x_i = \odot_{i=1}^n y_i$ if and only if 
$X_{(k)} = Y_{(k)}$ for $1 \le k \le n$.
}
\end{cor}
\begin{proof}
Follows from Theorem \ref{thm:div_chain_rep} and Proposition \ref{prop:div_chain_unique}.
\end{proof}

\begin{cor}
\label{cor:div_chan_can1}
\textsl{
Let $\sR$ be a B\'{e}zout domain. For $n \in \mathbb{N}$ and $x_1, \ldots, x_n, y_1, \ldots, y_n$ in $(\sR)$, we have
\begin{itemize}
\item[(i)] if $(\odot_{i=1}^{n}x_i) \odot [0] = (\odot_{i=1}^{n}y_i) \odot [0]$, then $\odot_{i=1}^{n}x_i = \odot_{i=1}^{n}y_i$.
\item[(ii)] if $\odot_{i=1}^{n}x_i = \odot_{i=1}^{n}y_i$, then the multiplicity of $[0]$ in the multiset $\{ x_i : i \in \langle n \rangle \}$ is equal to the multiplicity of $[0]$ in the multiset $\{ y_i : i \in \langle n \rangle \} $.
\end{itemize} 
}
\end{cor}
\begin{proof}
(i) By virtue of Theorem \ref{thm:div_chain_rep}, we may assume without loss of generality that $x_1 \le \cdots \le x_n, y_1 \le \cdots \le y_n$. Considering the divisor multichains $[0] \le x_1 \le \cdots \le x_n, [0] \le y_1 \le \cdots \le y_n$ in $(\sR)$, using Proposition \ref{prop:div_chain_unique}, we have that $x_i = y_i, 1\le i \le n$. Thus $\odot_{i=1}^{n}x_i = \odot_{i=1}^{n}y_i$.

(ii) Using part (i) successively, without loss of generality, we may assume that $[0]$ does not belong to the multiset $\{ x_i : i \in \langle n \rangle \}$ and thus $X_{(n)} \ne [0]$. By Theorem \ref{thm:div_chain_rep} and Proposition \ref{prop:div_chain_unique}, we have $X_{(n)} = Y_{(n)}.$ Hence $Y_{(n)} \ne [0]$ and we conclude that $[0]$ does not belong to the multiset $\{ y_i : i \in \langle n \rangle \}$.
\end{proof}

\begin{remark}
\label{rmrk:summary}
Let $\sR$ be an elementary divisor domain. We summarize the results proved above giving a picture of the structure of $\fU \fR (\mathcal{R})$.
\begin{itemize}
\item[(i)] By Corollary \ref{cor:div_chan_can1}, every non-zero element $X$ of $\fU \fR (\mathcal{R})$ has a {\it unique} representation as $\odot_{i=1}^n x_i$ for a multichain $x_1 \le \cdots \le x_n$ in $(\sR)$ that does not contain $[0]$. We call this the multichain representation of $X$ with respect to $(\sR)$. 
\item[(iii)] Theorem \ref{thm:fund_ufd_algo} provides an algorithm to compute the $\odot$-sum of two elements in $\fU \fR (\mathcal{R})$ whose multichain representations with respect to $(\sR)$ are provided.
\end{itemize}
  
\end{remark}

\begin{remark}
For a Euclidean domain $\sR$, the lattice operations on $(\sR)$ may be computed using the Euclidean algorithm successively and this may be used to implement the algorithm in Theorem \ref{thm:fund_ufd_algo} to obtain the $\odot$-sum of two elements of $\fU \fR (\mathcal{R})$. For the ring of integers $\mathbb{Z}$ and the univariate polynomial ring $\K[t]$ over a field $\K$ ($= \R, \C$), the Euclidean algorithm is quite efficient in practice.
\end{remark}

\begin{prop}
\label{prop:surj_supp_L1}
\textsl{
Let $\mathcal{R}$ be an elementary divisor domain. For $[X]$ in $\fU \fR (\mathcal{R})$, let $(\odot_{i=1}^m x_i )$ be the multichain representation of $[X]$ with respect to $(\sR)$. Then the map $S : \big( \fU \fR(\mathcal{R}), \odot \big) \rightarrow (\Z ^{+}, +)$ defined by $$S([0]) = 0, S([X]) := m \textrm{ for } [X] \ne [0],$$ and the map $\Delta : (\fU \fR (\sR), \odot) \rightarrow \big( (\sR), \cdot \big)$ defined by $$\Delta([0]) := [1], \Delta([X]) := x_1 x_2 \cdots x_m$$ are both surjective monoid homomorphisms.
}
\end{prop}
\begin{proof}
Note that $S([X] \odot [0]) = S([X]) = S([X]) + S([0])$. Similarly $\Delta([X] \odot [0]) = \Delta([X]) = \Delta([X]) \Delta([0])$. So we need only consider the cases where neither $[X]$ nor $[Y]$ is $[0]$.

Let the  multichain representations for $[X], [Y]$ in $\fU \fR (\sR)$ be $(\odot_{i=1}^m x_i)$, $(\odot_{j=1}^n y_j)$, respectively. Consider the multisubsets $\chi := \{ x_1, \ldots, x_m \}, \upsilon := \{ y_1, \ldots, y_n \}$ with their respective supports in $(\sR)$. Let $\zeta := \chi \sqcup \upsilon$. Note that $\chi_{(m)} \ne [0], \upsilon_{(n)} \ne [0]$. As $\zeta_{(m+n)} = \chi_{(m)} \vee \upsilon_{(n)} \ne [0]$, by Remark \ref{rmrk:summary}, the divisor multichain representation of $[X] \odot [ Y ]$ is given by $[\odot_{i=1}^{m+n} \zeta_{(i)}]$. Thus $S([X] \odot [Y]) = m+n = S([X]) + S([Y])$ which proves that $S$ is a monoid homomorphism. The surjectivity of $S$ follows from the fact that $S(\odot_{i=1}^n [1]) = n$ for all $n \in \mathbb{N}$.

Next we turn our attention to $\Delta$.\ By Proposition \ref{prop:prod_L1} and Theorem \ref{thm:m_chain_rep}, we see that $\Delta \big( [X] \odot [Y] \big) = \Delta \big( [ \odot_{i=1}^{m+n} \zeta_{(i)} ] \big) =  \prod_{i=1}^{m+n} \zeta_{(i)} = \big( \prod_{i=1}^m x_i \big) \big( \prod_{j=1}^n y_j \big) = \Delta \big( [X] \big) \Delta \big( [Y] \big)$ which proves that $\Delta$ is a monoid homomorphism. The surjectivity of $\Delta$ follows from the fact that $\Delta \big( [x] \big) = (x)$ for $x$ in $\sR$.
\end{proof}

\begin{remark}
\label{rmrk:k_ne_l}
Note that $\big( (\K[t]), \cdot \big)$ is not singly generated as a monoid. This is because there is no non-trivial polynomial that divides every polynomial in $\K[t]$. By Proposition \ref{prop:surj_supp_L1}, we conclude that $\big( \fU \fR (\K[t]), \odot \big)$ is not singly generated as a monoid and hence $$\fU \fR (\K[t]) \ncong \Z ^{+} \cong K_0(\K[t])^{+}.$$ This shows that the $\fU \fR$-functor is distinct from the $K_0$-functor. 
\end{remark}

\begin{remark}
\label{rmrk:UFD_lat}
Let $Q$ be a UFD and $\mathbb{P} := \{ (x) : x \in Q$ is irreducible$\}$. An element $x$ of $\big( (Q); \cdot \big)$ may be uniquely written as the product of elements of $\mathbb{P}$ upto permutation. For every $p \in \mathbb{P}$, we define a mapping $\pi_p : (Q) \rightarrow \Z ^{+} \cup \{ \infty \}$ by $\pi_p(x) = \sup \{ \alpha \in \Z ^+  : p^{\alpha} \le x \}$ for $x \ne (0)$ and $\pi_p\big( (0) \big) = \infty$.

In essence, $\pi_p$ gives the multiplicity of $p$ in the factorization of $x \ne (0)$. Note that $\pi_p$ is a lattice homomorphism, and $x = y$ in $(Q)$ if and only if $\pi_p(x) = \pi_p(y)$ for all $p$ in $\mathbb{P}$.
\end{remark}

\begin{thm}
\label{thm:algo_rank_identity}
Let $Q$ be a PID (and hence, an EDD and UFD) and $\mathbb{P} := \{ (x) : x \in Q$ is irreducible$\}$. Let $X := \{ x_1, \ldots, x_n \}, Y := \{ y_1, \ldots, y_n \}$ be finite multisubsets of equal cardinality with their respective supports contained in $(Q)$. Then $\odot_{i=1}^n x_i = \odot_{i=1}^n y_i$ if and only if $\pi_p(X) = \pi_p(Y)$ (as multisets) for every $p \in \mathbb{P}$.
\end{thm}

\begin{proof}
By Theorem \ref{thm:div_chain_rep}, we have $\odot_{i=1}^n x_i = \odot_{i=1}^n y_i$ if and only if  $X_{(k)} = Y_{(k)}$ for all $k \in \langle n \rangle$. By Remark \ref{rmrk:UFD_lat} and Remark \ref{rmrk:order_stat_latt}, for $1 \le k \le n$, we see that $X_{(k)} = Y_{(k)} \Longleftrightarrow \pi_p(X_{(k)}) = \pi_p(Y_{(k)})$ for every $p \in \mathbb{P}$ $ \Longleftrightarrow \pi_p(X)_{(k)} = \pi_p(Y)_{(k)}$ for every $p \in \mathbb{P}$. 

As $\Z ^{+}$ is totally ordered, any multisubset with support contained in $\Z ^+$ is a multichain and we conclude that $\pi_p(X)_{(k)} = \pi_p(Y)_{(k)}$ for every $p \in \mathbb{P}$ and every $k \in \langle n \rangle$ if and only if $\pi_p(X) = \pi_p(Y)$ for every $p \in \mathbb{P}$.
\end{proof}

In \S \ref{sec:ranked_rings}, we use Theorem \ref{thm:algo_rank_identity} to formulate Algorithm \ref{subsec:algo_rank_id} for generating rank identities.

\begin{remark}
Let $\K$ be a field. Note that $\K$ is an elementary divisor domain and $(\K)$ has two elements, $(0)$ and $(1)$. For $X \in M_{\mathrm{fin}}(\K)$, the multichain representation of $[X]$ with respect to $(\K)$ is of the form $\odot _{i=1}^n [1]$. Thus the mapping $S$ as defined in Proposition \ref{prop:surj_supp_L1} is an isomorphism from $\big( \fU \fR (\K); \odot \big)$ and $(\Z ^{+}; +)$ and coincides with the usual rank function on square matrices over $\K$. This also shows that the universal rank-system for $\K$ is given by the usual rank function.
\end{remark}

\subsection{The $\fU \fR$-monoid of Murray-von Neumann algebras}
\label{subsec:UR_mon_MvN}

\begin{lem}
\label{lem:ran_equiv}
\textsl{
Every element of a Murray-von Neumann algebra is $\sim_1$-equivalent to its range projection.
}
\end{lem}
\begin{proof}
Let $\fM$ be a Murray-von Neumann algebra. Let $x \in \fM$ and $x = up$ be the unitary polar decomposition of $x$, where $p$ is a positive element in $\fM$ and $u$ is a unitary element in $\fM$. Note that $p + N(p)$ is invertible in $\fM$ and $\big( p+N(p) \big)^{-1} p = R(p)$. Thus $p \sim_1 R(p)$. Since $R(p) = R(u^*x) = u^*R(x)u$, we conclude that $R(x) \sim_1 R(p) \sim_1 p \sim_1 x$.
\end{proof}

\begin{cor}
\label{cor:ran_equiv}
\textsl{
Let $\mathscr{R}$ be a finite von Neumann algebra. For $x, y \in \afr$, $x \sim_1 y$ if and only if $\fr _c(x) = \fr _c(y)$.
}
\end{cor}
\begin{proof}
If $x \sim_1 y$, it follows from Proposition \ref{prop:rank_prop}, (vii), that $\fr _c(x) = \fr _c(y)$. For the converse, let $\fr _c(x) = \fr _c(y)$. Thus $R(x)$ and $R(y)$ are $\sim_{MV}$-equivalent projections in $\mathscr{R}$. There is a unitary element $u$ in $\mathscr{R}$ such that $uR(x)u^* = R(y)$. Thus $R(x) \sim_1 R(y)$ in $\afr$ and it follows from Lemma \ref{lem:ran_equiv} that $x \sim_1 R(x) \sim_1 R(y) \sim_1 y$ in $\afr$.
\end{proof}

\begin{thm}
\label{thm:UR_mon_MvN}
\textsl{
Let $\mathscr{R}$ be a finite von Neumann algebra with center $\mathscr{C}$. Let $P(\mathscr{C})$ denote the set of projections in $\mathscr{C}$, and $\Z ^+[P(\mathscr{C})] \subset \mathscr{C}$ denote the set of $\Z ^+$-linear combinations of projections in $\mathscr{C}$. 
\begin{itemize}
\item[(i)] If $\mathscr{R}$ is of type $I_n$ (for $n \in \N$), then $\fU \fR (\afr) \cong \Z^{+}[P(\mathscr{C})]$;
\item[(ii)] If $\mathscr{R}$ is of type $II_1$, then $\fU \fR (\afr) \cong \mathscr{C}^{+}$.
\end{itemize}
}
\end{thm}
\begin{proof}
In view of Theorem \ref{thm:univ_rank_sys}, consider the monoid homomorphism $\varphi : \fU \fR (\afr) \to \mathscr{C}^+$ given by $\varphi(\mathrm{rank}_u(X)) = \fr _c(X)$, for $X \in M_{\mathrm{fin}}(\afr)$. 

Let $X_1 \in M_{k_1}(\afr), X_2 \in M_{k_2}(\afr)$ such that $\fr _c(X_1) = \fr _c(X_2)$. Without loss of generality, we may assume that $k_1 \le k_2$. Since $\fr _c(X_1 \odot {\bf 0}_{k_2 - k_1}) = \fr _c(X_2)$, by Corollary \ref{cor:ran_equiv}, $X_1 \odot {\bf 0}_{k_2 - k_1} \sim_{k_2} X_2$. Thus $\mathrm{rank}_u(X_1) = \mathrm{rank}_u(X_2)$. This shows that $\varphi$ is injective.
\vskip 0.1in

\noindent (i) If $\mathscr{R}$ is of type $I_n$, then it is $*$-isomorphic to $M_n(\mathscr{C})$ (see \cite[Theorem 6.6.5]{kadison-ringrose2}). Thus $\afr$ and $M_n(\afc)$ are isomorphic as Murray-von Neumann algebras (see \cite[Remark 4.19]{nayak_mvna}). Note that $\fU \fR (\afr) \cong \fU \fR(M_n(\afc)) \cong \fU \fR (\afc)$ by Corollary \ref{cor:UR_mon_morita}. For $x \in \afc$, we have $\fr _c(x) = R(x) \in \mathscr{C}$. For $X \in M_n(\afc)$, from Lemma \ref{lem:ran_equiv}, we have $X \sim_1 R(X)$. By the diagonalization theorem (see \cite[Corollary 3.20]{diag_kadison}), there are projections $e_1, \ldots, e_n \in \mathscr{C}$ such that $R(X)$ is unitarily equivalent to $\odot_{i=1}^n e_i$, or in particular, $R(X) \sim_n \odot_{i=1}^n e_i$. Thus $\fr _c(X) = \sum_{i=1}^n e_i \in \Z^{+}[P(\mathscr{C})]$. 

On the other hand, for $x \in \Z^{+}[P(\mathscr{C})]$, there are projections $e_1, \ldots, e_m \in \mathscr{C}$ such that $x = \sum_{i=1}^m e_i$. Thus $\fr _c(\odot_{i=1}^m e_i) = x$. In summary, the image of $\varphi$ is $\Z^{+}[P(\mathscr{C})]$.
\vskip 0.1in

\noindent (ii) Let $c \in \mathscr{C}^{+}$. Let $n \in \N$ such that $0 \le c \le n 1$. Since $0 \le \frac{1}{n} c \le 1$ and $\mathscr{R}$ is of type $II_1$, there is a projection $e \in \mathscr{R}$ such that $\fr _c(e) = \frac{1}{n}c$. Note that $\fr _c (\odot_{i=1}^n e) = c$ so that $\varphi([\odot_{i=1}^n e]) = c$. This shows that $\varphi$ is surjective.
\end{proof}

Since every Murray-von Neumann algebra is a direct sum of Murray-von Neumann algebras of types $I_n$ and $II_1$, Theorem \ref{thm:UR_mon_MvN} gives a complete characterization of the $\fU \fR$-monoid of Murray-von Neumann algebras.

\section{Rank Identities}
\label{sec:rank_id}

In this section, we define the notion of a {\it rank identity} in a ranked ring, and as an application of the results in \S \ref{sec:L_bezout}, we exhibit abstract rank identities in the univariate polynomial ring, $\K[t]$, and the ring of entire functions, $\sO(\C)$. In \S \ref{subsec:rank_id_idem}, we obtain several rank identities involving two idempotents in a ranked ring with an eye towards applications in \S \ref{sec:app} in the context of von Neumann regular rings and operator algebras.

Throughout this section, we use $\mathcal{R}$ to denote a ring and $(\sM, +)$ to denote a commutative monoid. Appropriate subscripts are introduced when multiple rings or commutative monoids are being discussed. The additive identities in both $\sR$ and $\sM$ are denoted by $0$, and the multiplicative identity in $\sR$ by $1$. Their usage is clear from the context.

\begin{definition}
Let $\sR$ be a $\rho$-ranked ring. A {\it rank identity} or {\it rank equality} in $(\sR, \rho)$ is an equation of the form, $\rho(X) = \rho(Y)$, for $X, Y \in M_{\mathrm{fin}}(R)$.
\end{definition}

\begin{prop}
\label{prop:rank_id_fund}
Let $\sR_1$ be a ring, and $\sR_2$ be a $\rho$-ranked ring. Let $\Phi : \sR_1 \to \sR_2$ be a ring homomorphism. For $X, Y \in M_{\mathrm{fin}}(\sR_1)$, if $\mathrm{rank}_{u}(X) = \mathrm{rank}_{u}(Y)$, then $\rho(\Phi(X)) = \rho(\Phi(Y))$.
\end{prop}
\begin{proof}
This is straightforward from the definition of $\fU \fR (\mathcal{R})$ and the properties \ref{def:ranked}(a), (b), (c) for $\rho$.
\end{proof}

\subsection{Rank identities in the univariate polynomial ring over a field}
\label{subsec:rank_id}

\begin{thm}
\label{thm:rank_id_pol}
\textsl{
Let $\mathcal{R}$ be a $\rho$-ranked ring with a field $\K$ in its center. For non-zero polynomials $p_1, \ldots, p_n$, $q_1, \ldots, q_n$ in $\K[t]$, if
\begin{equation}
\label{eqn:wedge_vee}
\wedge_{1\le i_1 < \cdots < i_k  \le n} ([p_{i_1}] \vee \cdots \vee [p_{i_k}]) = \wedge_{1 \le i_1 < \cdots < i_k \le n} ([q_{i_1}] \vee \cdots \vee [q_{i_k}] ), 1 \le k \le n,
\end{equation}
then for $x$ in $\mathcal{R}$, we have
$$\sum_{i=1}^n \rho \big( p_i(x) \big) = \sum_{j=1}^n \rho \big( q_j(x) \big).$$
}
\end{thm}
\begin{proof}
Note that if equation (\ref{eqn:wedge_vee}) holds, then by Corollary \ref{cor:odot_equal_prod}, we have $\odot_{i=1}^n [p_i] = \odot_{i=1}^n [q_i]$ in $\fU \fR (\K[t])$. As $\K$ is in the center of $\mathcal{R}$, the element $x$ commutes with $\K$. By the universal property of the univariate polynomial ring over $\K$, the evaluation map $\sE : \K[t] \rightarrow \mathcal{R}$ where $\sE (t) = x$, defines a ring homomorphism sending $p$ to $p(x)$. By Proposition \ref{prop:rank_id_fund} and property \ref{def:ranked} (a) for $\rho$, we have $$\sum_{i=1}^n \rho \big( \sE (p_i) \big) = \sum_{i=1}^n \rho \big( \sE (q_i) \big).$$
\end{proof}

\begin{remark}
Note that by Corollary \ref{cor:div_chan_can1}, (ii), for non-zero polynomials $p_1, \ldots, p_m, q_1, \ldots, q_n$ in $\K[t]$, if $\odot_{i=1}^m [p_i] = \odot_{i=1}^n [q_i]$, then $m=n$.
\end{remark}

\begin{prop}[partial converse to Theorem \ref{thm:rank_id_pol}]
\label{prop:exhaustive}
\textsl{
Let $\sM$ be a cancellative monoid and $\K$ be a field. Let $p_1, \ldots, p_m, q_1, \ldots, q_n$ be non-zero polynomials in $\K[t]$ such that for any $(\sM, \rho)$-ranked $\K$-algebra $\sR$ we have
\begin{equation}
\label{eqn:rank_id_conv}
\sum_{i=1}^m \rho \big( p_i(x) \big) = \sum_{j=1}^n \rho \big( q_j(x) \big) , \forall x \in \sR.
\end{equation} Then $m = n$  and 
\begin{equation}
\label{eqn:wedge_vee_conv}
\wedge_{1\le i_1 < \cdots < i_k  \le n} ([p_{i_1}] \vee \cdots \vee [p_{i_k}]) = \wedge_{1 \le i_1 < \cdots < i_k \le n} ([q_{i_1}] \vee \cdots \vee [q_{i_k}] ), 1 \le k \le n.
\end{equation}
}
\end{prop}
\begin{proof}
Using Theorem \ref{thm:rank_id_pol}, without loss of generality,  we may assume that $[p_1] \le  [p_2] \le  \cdots \le  [p_m]$, and $[q_1] \le [q_2] \le \cdots \le [q_n]$. We show that $m = n$ and $[p_i] = [q_i]$ for $i = 1, 2, \ldots, n$.

Let, if possible, $[p_1] \ne [q_1]$. Without loss of generality, we may assume that there is a non-constant polynomial $r$ in $\K [t]$ such that $r | p_1$ and $r \nmid q_1$. Choose $k \in \N$ and $A \in M_k(\K) \subseteq M_k(\sR)$ with minimal polynomial $r$. Then $p_1(A) = p_2(A) = \cdots = p_m(A) = 0$. On the other hand, $q_1(A) \ne 0$ so that the rank identity (\ref{eqn:rank_id_conv}) fails to hold for $A$ which is an element of the ranked $\K$-algebra $M_k(\sR)$. This leads to a contradiction and hence $[p_1] = [q_1]$. The rest of the claim follows from an easy induction argument using the cancellation property of $\sM$.
\end{proof}

\subsubsection{An Algorithm to Generate Rank Identities\\}
\label{subsec:algo_rank_id}
Let $\K$ be a field.
\begin{itemize}
\item[1.] Start with distinct irreducible polynomials $p_1, \ldots, p_n$ in $\K[t]$.
\item[2.] Construct an $m \times n$ matrix $\Lambda = (\lambda_{ij}), (i, j) \in \langle m \rangle \times \langle n \rangle$ with entries from $\mathbb{Z}_+$. Define $q_i := \prod_{j=1}^n p_j^{\lambda_{ij}}$ for $1 \le i \le m$. 
\item[3.] Shuffle the entries in each column of $\Lambda$ to obtain an $m \times n$ matrix $\mu = (\mu_{ij}),(i, j) \in \langle m \rangle \times \langle n \rangle$ with entries from $\mathbb{Z}_{+}$. Define $r_i := \prod_{j=1}^n p_j^{\mu_{ij}}$ for $1 \le i \le m$. 
\end{itemize}
We say that $\mu$ is an {\it intra-column shuffle} of $\Lambda$ and vice versa. We say that the tuple of polynomials $(q_1, \ldots, q_n)$ is associated with the matrix $\Lambda$, and $(r_1, \ldots, r_n)$ is associated with $\mu$ (with respect to $\{ p_1, \ldots, p_n \}$). Note that $\K[t]$ is a PID and hence a UFD. By Theorem \ref{thm:algo_rank_identity} and Theorem \ref{thm:rank_id_pol}, for an element $x$ of a ranked $\K$-algebra, we have the following rank identity :
\begin{equation}
\label{algo:rank_id}
\sum_{i=1}^m \rho \big( q_i(x) \big) = \sum_{i=1}^m \rho \big( r_i(x) \big).
\end{equation}

\begin{thm}
\label{thm:rank_entire_id}
\textsl{
Let $\mathcal{R}$ be a $\rho$-ranked complex Banach algebra. For entire functions $f_1, \ldots, f_n$, $g_1, \ldots, g_n$ in $\mathcal{O}(\mathbb{C})$ if 
\begin{equation}
\label{eqn:wedge_vee_entire}
\wedge_{1\le i_1\le \cdots \le i_k  \le n} ([f_{i_1}] \vee \cdots \vee [f_{i_k}]) = \wedge_{1\le i_1\le \cdots \le i_k  \le n} ([g_{i_1}] \vee \cdots \vee [g_{i_k}] ), 1 \le k \le n,
\end{equation}
then for $x$ in $\mathcal{R}$, we have
$$\sum_{i=1}^n \rho \big( f_i(x) \big) = \sum_{i=1}^n \rho \big( g_i(x) \big).$$
}
\end{thm}
\begin{proof}
The proof is almost identical to that of Theorem \ref{thm:rank_id_pol} with the polynomial function calculus replaced by the holomorphic function calculus. If equation (\ref{eqn:wedge_vee_entire}) holds, then by Corollary \ref{cor:odot_equal_prod}, we have $\odot_{i=1}^n [f_i] = \odot_{i=1}^n [g_i]$ in $\fU \fR \big( \mathcal{O}(\mathbb{C}) \big)$. By the holomorphic function calculus for complex Banach algebras, there is a ring homomorphism $\sE : \sO (\mathbb{C}) \rightarrow \mathcal{R}$ such that $\sE(f) = f(x)$ for $f \in \mathcal{O}(\C)$. This induces a monoid homomorphism $\fU \fR (\sE) : \fU \fR \big( \sO(\C) \big) \rightarrow \fU \fR(\sR)$ and as $\odot_{i=1}^n [f_i] = \odot_{i=1}^n [g_i]$ in $\fU \fR \big( \sO (\C) \big)$, we have $\odot_{i=1}^n [\sE(f_i)] = \odot_{i=1}^n [\sE(g_i)]$ in $\fU \fR \big( \sO(\C) \big)$. By Proposition \ref{prop:rank_id_fund}, $$\sum_{i=1}^n \rho \big( \sE(f_i) \big) = \sum_{i=1}^n \rho \big( \sE(g_i) \big).$$
\end{proof}

\begin{example}
As $\cos ^2 (\theta) + \sin ^2(\theta) = 1$ for all $\theta \in \mathbb{C}$,  the entire functions $\cos (z), \sin (z)$ are coprime in $\mathcal{O}(\mathbb{C})$. Let $f_1 = \cos (z), f_2 = \sin(z)$ and $g_1 = 1, g_2 = \sin(2z)$ be elements of $\mathcal{O}(\mathbb{C})$. We have $$[f_1] \wedge [f_2] = [g_1] \wedge [g_2] = [1], $$ $$[f_1] \vee [f_2] = [g_1] \vee [g_2] = [\sin(z)\cos(z)].$$ Thus for $T$ in a complex Banach algebra with a rank-system $\rho$, we have the following rank identity: $$\rho \big( \cos(T) \big) + \rho \big( \sin(T) \big) = \rho \big( 1 \big) + \rho \big( \sin(2T) \big).$$
In particular, this identity holds in finite von Neumann algebras for the center-valued rank as defined in \S \ref{subsec:center_valued}.
\end{example}

\subsection{Examples of rank identities in free associative algebras over a field}
\label{subsec:free_assoc}

In the discussion so far, we have identified the importance of a natural language to develop a framework to study rank identities. To this end, in \S \ref{sec:L_bezout} we determined the $\fU \fR$-monoid of elementary divisor domains. For a field $\K$, let $\K \langle t_1, \ldots, t_n \rangle$ denote the free associative algebra over $\K$ in the indeterminates $t_1, \ldots, t_n$. We consider the problem of determining the structure of the $\fU \fR$-monoid of $\K \langle t_1, \ldots, t_n \rangle$ to be the natural next step in keeping with the spirit of the program. Although this is not accomplished in this article, we show some examples of rank identities in the discussion below to serve as motivation for future work.

\begin{prop} 
\label{prop:free_assoc_id}
\textsl{
In $\fU \fR (\K \langle t_1, t_2 \rangle )$, we have 
\begin{itemize}
\item[(i)] $\mathrm{rank}_u(1 + t_1 t_2) \odot \mathrm{rank}_u(1) = \mathrm{rank}_u(1 + t_2 t_1) \odot \mathrm{rank}_u(1)$,
\item[(ii)] $\mathrm{rank}_u(t_1(1 + t_2 t_1)) \odot \mathrm{rank}_u(1) = \mathrm{rank}_u(t_1) \odot \mathrm{rank}_u(1 + t_2 t_1)$.
\item[(iii)] $
\mathrm{rank}_u(t_1) =
\mathrm{rank}_u \Big( \begin{bmatrix}
t_1 t_2 & t_1 \\
0 & 0
\end{bmatrix} =
\mathrm{rank}_u \Big( \begin{bmatrix}
t_1 & t_1 t_2 \\
0 & 0
\end{bmatrix} \Big).
$
\item[(iv)] $
\mathrm{rank}_u(t_1) =
\mathrm{rank}_u \Big( \begin{bmatrix}
t_2 t_1 & 0 \\
t_1 & 0
\end{bmatrix} =
\mathrm{rank}_u \Big( \begin{bmatrix}
t_1 & 0 \\
t_2 t_1 & 0
\end{bmatrix} \Big).
$
\end{itemize}
}
\end{prop}
\begin{proof}
By Lemma \ref{lem:invertible_M2}, for $r \in \K \langle t_1, t_2 \rangle$ the following matrices 
$$\begin{bmatrix}
r & 1\\
1 & 0
\end{bmatrix},
\begin{bmatrix}
1 & r\\
0 & 1
\end{bmatrix} ,
\begin{bmatrix}
1 & 0\\
r & 1
\end{bmatrix},
\begin{bmatrix}
0 & 1\\
1 & r
\end{bmatrix}$$
are invertible in $M_2(\K \langle t_1, t_2 \rangle)$.
It is straightforward to check the matrix computations below.
\begin{equation}
\label{eqn:free_assoc_two1}
\begin{bmatrix}
1 + t_2 t_1 & 0 \\
0 & 1
\end{bmatrix}
= 
\begin{bmatrix}
t_2 & 1\\
1 & 0
\end{bmatrix}
\begin{bmatrix}
1 & t_1\\
0 & 1
\end{bmatrix} 
\begin{bmatrix}
1 + t_1 t_2 & 0 \\
0 & 1
\end{bmatrix}
\begin{bmatrix}
1 & 0\\
-t_2 & 1
\end{bmatrix}
\begin{bmatrix}
-t_1 & 1\\
1 & 0
\end{bmatrix},
\end{equation}

\begin{equation}
\label{eqn:free_assoc_two2}
\begin{bmatrix}
t_1(1 + t_2 t_1) & 0\\
0 & 1
\end{bmatrix} = 
\begin{bmatrix}
t_1 & -1\\
1 & 0
\end{bmatrix} 
\begin{bmatrix}
-t_2 & 1\\
1 & 0
\end{bmatrix}
\begin{bmatrix}
1 + t_2 t_1 & 0\\
0 & t_1
\end{bmatrix}
\begin{bmatrix}
1 & 1\\
1 & 0
\end{bmatrix}
\begin{bmatrix}
-(1 + t_2 t_1) & 1 \\
1 & 0
\end{bmatrix},
\end{equation}

\begin{equation}
\label{eqn:free_assoc_two3}
\begin{bmatrix}
t_1 & 0\\
0 & 0
\end{bmatrix} = 
\begin{bmatrix}
t_1 t_2 & t_1 \\
0 & 0
\end{bmatrix}
\begin{bmatrix}
0 & 1 \\
1 & -t_2
\end{bmatrix} = 
\begin{bmatrix}
t_1 & t_1 t_2 \\
0 & 0
\end{bmatrix}
\begin{bmatrix}
1 & -t_2\\
0 & 1
\end{bmatrix},
\end{equation}

\begin{equation}
\label{eqn:free_assoc_two4}
\begin{bmatrix}
t_1 & 0\\
0 & 0
\end{bmatrix} = 
\begin{bmatrix}
0 & 1\\
1 & -t_2
\end{bmatrix}
\begin{bmatrix}
t_2 t_1 & 0 \\
t_1 & 0
\end{bmatrix} = 
\begin{bmatrix}
1 & 0\\
-t_2 & 0
\end{bmatrix}
\begin{bmatrix}
t_1 & 0 \\
t_2 t_1 & 0
\end{bmatrix}.
\end{equation}
The assertions (i), (ii), (iii), respectively, follow from equation (\ref{eqn:free_assoc_two1}), equation (\ref{eqn:free_assoc_two2}), equation (\ref{eqn:free_assoc_two3}), respectively.
\end{proof}

\begin{thm}
\label{thm:rank_non_pol_id}
\textsl{
Let $\mathcal{R}$ be a $(\sM, \rho)$-ranked ring and $\K$ be a subfield of the center of $\mathcal{R}$. For non-commutative polynomials $p_1, \ldots, p_k$, $q_1, \ldots, q_{\ell}$ in $\K \langle t_1, \ldots, t_n \rangle$, if $\odot_{i=1}^k [p_i] = \odot_{i=1}^{\ell} [q_i]$ in $\fU \fR (\K \langle t_1, \ldots, t_n \rangle)$
then for $x_1, \ldots, x_n$ in $\mathcal{R}$, we have
$$\sum_{i=1}^k \rho \big( p_i(x_1, \ldots, x_n) \big) = \sum_{i=1}^{\ell} \rho \big( q_i(x_1, \ldots, x_n) \big).$$
}
\end{thm}
\begin{proof}
Note that the elements $x_1, \ldots, x_n$ commute with $\K$ as $\K$ is in the center of $\mathcal{R}$. The proof is similar to that of Theorem \ref{thm:rank_id_pol} using the universal property of the free associative algebra over $\K$ in $n$ indeterminates.
\end{proof}

\subsection{Examples of rank identities involving idempotents}
\label{subsec:rank_id_idem}

Let $\sR$ be a ring. In this subsection, we prove several rank identities involving idempotents in $\mathcal{R}$ with a view towards applications in \S \ref{sec:app}. As the structure of the proofs is quite similar, we outline the basic strategy to streamline the process. 
\vskip 0.04in
\textbf{Step 1.}  We consider matrices $A, B$  in $M_n(\mathcal{R})$ and prove that they are invertible by explicitly finding matrices $A', B'$ in $M_n(\mathcal{R})$ such that $AA' = A'A =I_n = B'B = BB'$.
 \vskip 0.04in
\textbf{Step 2.} For matrices $X, Y$ in $M_n(\mathcal{R})$, we show that $AXB = Y$ which implies that $\mathrm{rank}_u(X) = \mathrm{rank}_u(Y)$.
\vskip 0.04in

Steps $1$, $2$ in the proofs involve direct verification via matrix multiplication for the appropriate choices of $A, B, A', B'$. The process of discovery of these matrices in $M_n(\mathcal{R})$ is not explicitly worked out for the sake of brevity and hence their choice may sometimes look mysterious. It may help to recognize that the key constitutive steps involve elementary row and column operations on $X$ to arrive at $Y$. These elementary operations are neatly packaged by combining the successive row operations in $A$ and the successive column operations in $B$.

\begin{lem}
\label{lem:fund_id}
For elements $x, y, e$ in $\mathcal{R}$ with $e$ being idempotent, we have :
\begin{itemize}
\item[(i)] $
\mathrm{rank}_u \Big( \begin{bmatrix}
xe & 0\\
y(1-e) & 0
\end{bmatrix} \Big) = 
\mathrm{rank}_u(xe) \odot \mathrm{rank}_u \big( y(1-e) \big)$,
\item[(ii)]
$
\mathrm{rank}_u \Big( \begin{bmatrix}
ex & (1-e)y\\
0 & 0
\end{bmatrix} \Big) = 
\mathrm{rank}_u(ex) \odot \mathrm{rank}_u \big( (1-e)y \big)$,
\item[(iii)] $
\mathrm{rank}_u \Big( \begin{bmatrix}
e & 0\\
x & 0
\end{bmatrix} \Big) = 
\mathrm{rank}_u(e) \odot \mathrm{rank}_u \big( x(1-e) \big)$,
\item[(iv)] $
\mathrm{rank}_u \Big( \begin{bmatrix}
e & x\\
0 & 0
\end{bmatrix} \Big)  = 
\mathrm{rank}_u(e) \odot \mathrm{rank}_u \big( (1-e)x \big)$,
\end{itemize}
\end{lem}

\begin{proof}$\phantom{}$
\begin{center}
\underline{{\it Step 1.}}
\end{center}
\begin{align*}
\textrm{(for (i)-(iv) )}
\begin{bmatrix}
1 & 0\\
0 & 1
\end{bmatrix} =
\begin{bmatrix}
e & 1-e\\
1+e & -e
\end{bmatrix}^2
= 
\begin{bmatrix}
e & 1+e\\
1-e & -e
\end{bmatrix}^2 
\end{align*}
\begin{align*}
\textrm{(for (iii),(iv) )} 
\begin{bmatrix}
1 & 0\\
0 & 1
\end{bmatrix} = 
\begin{bmatrix}
1 & 0\\
x & 1
\end{bmatrix}
\begin{bmatrix}
1 & 0 \\
-x & 1 
\end{bmatrix} = 
\begin{bmatrix}
1 & 0\\
-x & 1
\end{bmatrix}
\begin{bmatrix}
1 & 0 \\
x & 1 
\end{bmatrix}
\end{align*}

\begin{center}
\underline{{\it Step 2.}}
\end{center}
\begin{align*}
\textrm{(for (i))} \begin{bmatrix}
xe & 0\\
y(1-e) & 0
\end{bmatrix}
\begin{bmatrix}
e & 1-e\\
1+e & -e
\end{bmatrix} = 
\begin{bmatrix}
xe & 0\\
0 & y(1-e)
\end{bmatrix}
\end{align*}

\begin{align*}
\textrm{(for (ii)) }
\begin{bmatrix}
e & 1+e\\
1-e & -e
\end{bmatrix}
\begin{bmatrix}
ex & (1-e)y\\
0 & 0
\end{bmatrix} = 
\begin{bmatrix}
ex & 0\\
0 & (1-e)y
\end{bmatrix}
\end{align*}
\begin{align*}
\textrm{(for (iii)) } \begin{bmatrix}
1 & 0\\
-x & 1
\end{bmatrix}
\begin{bmatrix}
e & 0\\
x & 0
\end{bmatrix}
\begin{bmatrix}
e & 1-e\\
1+e & -e
\end{bmatrix} = 
\begin{bmatrix}
e & 0\\
0 & x(1-e)
\end{bmatrix}
\end{align*}
\begin{align*}
\textrm{(for (iv)) } \begin{bmatrix}
e & 1+e\\
1-e & -e
\end{bmatrix}
\begin{bmatrix}
e & x\\
0 & 0
\end{bmatrix}
\begin{bmatrix}
1 & -x\\
0 & 1
\end{bmatrix} = 
\begin{bmatrix}
e & 0\\
0 & (1-e)x
\end{bmatrix}
\end{align*}

Thus (i), (ii), (iii), (iv) follow from property \ref{def:ranked}(a) for $\mathrm{rank}_u$.
\end{proof}

\begin{cor}
\label{cor:fund_id_cor_1}
Let $\sR$ be a ring. For idempotents $e, f$ in $\mathcal{R}$, we have,
\begin{itemize}
\item[(i)] $\mathrm{rank}_u(e) \odot \mathrm{rank}_u \big( f(1-e) \big) = \mathrm{rank}_u(f)\odot \mathrm{rank}_u \big( e(1-f) \big)$;
\item[(ii)] $\mathrm{rank}_u(e) \odot \mathrm{rank}_u \big( (1-e)f \big) = \mathrm{rank}_u(f) \odot \mathrm{rank}_u \big( (1-f)e \big)$.
\end{itemize}

\end{cor}

\begin{proof}$\phantom{}$
\begin{center}
\underline{{\it Step 1.}}
\end{center}
\begin{align*}
\textrm{(for (i), (ii) )} \begin{bmatrix}
1 & 0\\
0 & 1
\end{bmatrix}=
\begin{bmatrix}
0 & 1\\
1 & 0
\end{bmatrix}^2 
\end{align*}

\begin{center}
\underline{{\it Step 2.}}
\end{center}
\begin{align*}
\textrm{(for (i)) }\begin{bmatrix}
e & f\\
0 & 0
\end{bmatrix} \begin{bmatrix}
0 & 1\\
1 & 0
\end{bmatrix}=
\begin{bmatrix}
f & e\\
0 & 0
\end{bmatrix}
\end{align*}
\begin{align*}
\textrm{(for (ii)) }\begin{bmatrix}
0 & 1\\
1 & 0
\end{bmatrix}\begin{bmatrix}
e & 0\\
f & 0
\end{bmatrix} =
\begin{bmatrix}
f & 0\\
e & 0
\end{bmatrix}
\end{align*}

Thus $\mathrm{rank}_u \Big( \begin{bmatrix}
e & f\\
0 & 0
\end{bmatrix} \Big) = 
\mathrm{rank}_u \Big( \begin{bmatrix}
f & e\\
0 & 0
\end{bmatrix} \Big)$,  
$\mathrm{rank}_u \Big( \begin{bmatrix}
e & 0\\
f & 0
\end{bmatrix} \Big) = 
\mathrm{rank}_u \Big( \begin{bmatrix}
f & 0\\
e & 0
\end{bmatrix} \Big)$. As a result, (i), (ii) follow from Lemma \ref{lem:fund_id}, (iii), (iv), respectively.
\end{proof}

\begin{lem}
\label{lem:diff_rank_id}
\textsl{
Let $\sR$ be a ring. For idempotents $e, f$ in $\mathcal{R}$, we have,
\begin{align*}
&\mathrm{rank}_u(e) \odot \mathrm{rank}_u(f) \odot \mathrm{rank}_u(e-f) \\
= &\mathrm{rank}_u(e) \odot \mathrm{rank}_u(f) \odot \mathrm{rank}_u \big( e(1-f) \big) \odot \mathrm{rank}_u \big( (1-e)f \big) \\
= & \mathrm{rank}_u(e) \odot \mathrm{rank}_u(f) \odot \mathrm{rank}_u \big( (1-f)e \big) \odot \mathrm{rank}_u \big( f(1-e) \big)
\end{align*}
}
\end{lem}

\begin{proof} (i) For $X, Y $ in $M_4(\mathcal{R})$ as below, $$
X := \begin{bmatrix}
e & f & 0 & 0\\
0 & 0 & 0 & 0\\
0 & 0 & e & 0\\
0 & 0 & f & 0
\end{bmatrix},
Y := \begin{bmatrix}
-e & 0 & 0 & 0\\
0 & f & 0 & 0\\
0 & 0 & e-f & 0\\
0 & 0 & 0 & 0
\end{bmatrix},
$$
we first prove that $\mathrm{rank}_u(X) = \mathrm{rank}_u(Y)$.

\begin{center}
\underline{{\it Step 1.}}
\end{center}
\begin{align*}
\begin{bmatrix}
1 & 0 & 0 & 0\\
0 & 1 & 0 & 0\\
0 & 0 & 1 & 0\\
0 & 0 & 0 & 1
\end{bmatrix} &=
\begin{bmatrix}
0 & 0 & 1 & 0\\
f & 0 & 0 & 1\\
1-f & 0 & 1 & -1\\
0 & 1 & 0 & 0
\end{bmatrix}
\begin{bmatrix}
-1 & 1 & 1 & 0\\
0 & 0 & 0 & 1\\
1 & 0 & 0 & 0\\
f & 1-f & -f & 0
\end{bmatrix}\\
&= 
\begin{bmatrix}
-1 & 1 & 1 & 0\\
0 & 0 & 0 & 1\\
1 & 0 & 0 & 0\\
f & 1-f & -f & 0
\end{bmatrix}
\begin{bmatrix}
0 & 0 & 1 & 0\\
f & 0 & 0 & 1\\
1-f & 0 & 1 & -1\\
0 & 1 & 0 & 0
\end{bmatrix}.\\
\begin{bmatrix}
1 & 0 & 0 & 0\\
0 & 1 & 0 & 0\\
0 & 0 & 1 & 0\\
0 & 0 & 0 & 1
\end{bmatrix} &=
\begin{bmatrix}
1 & 0 & 1& 0\\
0 & 1 & -1 & 0\\
-e & 0 & 1-e & 0\\
0 & 0 & 0 & 1
\end{bmatrix}
\begin{bmatrix}
1-e & 0 & -1& 0\\
e & 1 & 1 & 0\\
e & 0 & 1 & 0\\
0 & 0 & 0 & 1
\end{bmatrix}\\
&=
\begin{bmatrix}
1-e & 0 & -1& 0\\
e & 1 & 1 & 0\\
e & 0 & 1 & 0\\
0 & 0 & 0 & 1
\end{bmatrix}
\begin{bmatrix}
1 & 0 & 1& 0\\
0 & 1 & -1 & 0\\
-e & 0 & 1-e & 0\\
0 & 0 & 0 & 1
\end{bmatrix}.
\end{align*}

\begin{center}
\underline{{\it Step 2.}}
\end{center}
$$\begin{bmatrix}
0 & 0 & 1 & 0\\
f & 0 & 0 & 1\\
1-f & 0 & 1 & -1\\
0 & 1 & 0 & 0
\end{bmatrix}
\begin{bmatrix}
e & f & 0 & 0\\
0 & 0 & 0 & 0\\
0 & 0 & e & 0\\
0 & 0 & f & 0
\end{bmatrix}
\begin{bmatrix}
1 & 0 & 1& 0\\
0 & 1 & -1 & 0\\
-e & 0 & 1-e & 0\\
0 & 0 & 0 & 1
\end{bmatrix} = \begin{bmatrix}
-e & 0 & 0 & 0\\
0 & f & 0 & 0\\
0 & 0 & e-f & 0\\
0 & 0 & 0 & 0
\end{bmatrix}.$$\\

Using property \ref{def:ranked}(a) for $\mathrm{rank}_u$ and Corollary \ref{cor:fund_id_cor_1}, we see that 
\begin{align*}
&\phantom{=}\mathrm{rank}_u(e) \odot \mathrm{rank}_u(f) \odot \mathrm{rank}_u(e-f)\\
&=
\mathrm{rank}_u \Big(
\begin{bmatrix}
 e & f\\
 0 & 0
\end{bmatrix}  \Big) \odot
\mathrm{rank}_u \Big( 
\begin{bmatrix}
e & 0\\
f & 0
\end{bmatrix} \Big)\\
&= \mathrm{rank}_u(e) \odot \mathrm{rank}_u \big( (1-e)f \big) \odot \mathrm{rank}_u(f) \odot \mathrm{rank}_u \big( e(1-f) \big).
\end{align*}

The second equality in the statement of the lemma follows by exchanging $e$ and $f$.
\end{proof}

\begin{prop}
\label{prop:aux_rank_id_1}
\textsl{
Let $\sR$ be a ring. For idempotents $e, f$ in $\mathcal{R}$, we have,
\begin{equation}
\mathrm{rank}_u \Big(
\begin{bmatrix}
e & f\\
f & 0
\end{bmatrix} \Big) = 
\mathrm{rank}_u(f) \odot \mathrm{rank}_u(f) \odot \mathrm{rank}_u \big( (1-f)e(1-f) \big).
\end{equation}
}
\end{prop}
\begin{proof}
Let $f_2 = \begin{bmatrix}
f & 0\\
0 & 0
\end{bmatrix}, e_2 = \begin{bmatrix}
e & 0\\
f & 0
\end{bmatrix}.$
Note that $f_2$ is an idempotent in $M_2(\mathcal{R})$. It is easy to see that $\mathrm{rank}_u(f_2) = \mathrm{rank}_u(f)$ and,
$$
\mathrm{rank}_u \Big( \begin{bmatrix}
e & f\\
f & 0
\end{bmatrix}\Big) = 
\mathrm{rank}_u \Big( \begin{bmatrix}
e & f\\
f & 0
\end{bmatrix} \odot {\bf 0}_2\Big) = \mathrm{rank}_u \Big( \begin{bmatrix}
f_2 & e_2\\
0_2 & 0_2
\end{bmatrix} \Big)$$

By Lemma \ref{lem:fund_id}, we have 

\begin{align*}
\mathrm{rank}_u \Big( \begin{bmatrix}
f_2 & e_2\\
0_2 & 0_2
\end{bmatrix} \Big) &= \mathrm{rank}_u(f_2) \odot \mathrm{rank}_u \big( (1-f_2)e_2 \big)\\
&= \mathrm{rank}_u(f) \odot \mathrm{rank}_u\Big(\begin{bmatrix}
(1-f)e & 0\\
f & 0
\end{bmatrix} \Big)\\
& = \mathrm{rank}_u(f) \odot \mathrm{rank}_u(f) \odot \mathrm{rank}_u \big( (1-f)e(1-f) \big).
\end{align*}
\end{proof}

\begin{thm}
\label{thm:rank_sub}
\textsl{
Let $\sR$ be a ring in which the central element $2$ is invertible. Then for idempotents $e, f$ in $\mathcal{R}$, we have
\begin{align*}
&\mathrm{rank}_u(e) \odot \mathrm{rank}_u(f) \odot \mathrm{rank}_u(e+f)\\
=&\mathrm{rank}_u(e) \odot \mathrm{rank}_u(f) \odot \mathrm{rank}_u(f) \odot \mathrm{rank}_u \big( (1-f)e(1-f) \big)\\
=&\mathrm{rank}_u(e) \odot \mathrm{rank}_u(e) \odot \mathrm{rank}_u(f) \odot \mathrm{rank}_u \big( (1-e)f(1-e) \big)
\end{align*}
}
\end{thm}

\begin{proof}
For $X, Y$ in $M_3(\sR)$ as below, $$
X := \begin{bmatrix}
e & 0 & 0\\
0 & f & 0\\
0 & 0 & -e-f
\end{bmatrix},
Y := \begin{bmatrix}
e & 0 & 0\\
0 & e & f\\
0 & f & 0
\end{bmatrix},$$
we prove that $\mathrm{rank}_u(X) = \mathrm{rank}_u(Y)$.

\begin{center}
\underline{{\it Step 1.}}
\end{center}

\begin{align*}
\begin{bmatrix}
1 & 0 & 0\\
0 & 1 & 0\\
0 & 0 & 1\\
\end{bmatrix}&=
2^{-1}
\begin{bmatrix}
1 & 0 & 0\\
1 & 2 & 2\\
f & -(1-f) & f
\end{bmatrix} \cdot 
\begin{bmatrix}
2 & 0 & 0\\
f & f & -2\\
-(1+f) & 1-f & 2
\end{bmatrix}\\
&=
\begin{bmatrix}
2 & 0 & 0\\
f & f & -2\\
-(1+f) & 1-f & 2
\end{bmatrix}
\cdot 2^{-1}
\begin{bmatrix}
1 & 0 & 0\\
1 & 2 & 2\\
f & -(1-f) & f
\end{bmatrix}.\\
\begin{bmatrix}
1 & 0 & 0\\
0 & 1 & 0\\
0 & 0 & 1\\
\end{bmatrix}&=
\begin{bmatrix}
1+e & e-1 & 0\\
e & e-2 & 1\\
e & e-2 & 0
\end{bmatrix} \cdot 
2^{-1}
\begin{bmatrix}
2-e & 0 & e-1\\
e & 0 & -(1+e)\\
0 & 2 & -2
\end{bmatrix}\\
&=
2^{-1}
\begin{bmatrix}
2-e & 0 & e-1\\
e & 0 & -(1+e)\\
0 & 2 & -2
\end{bmatrix} \cdot
\begin{bmatrix}
1+e & e-1 & 0\\
e & e-2 & 1\\
e & e-2 & 0
\end{bmatrix}.
\end{align*}

\begin{center}
\underline{{\it Step 2.}}
\end{center}$$
2^{-1}
\begin{bmatrix}
1 & 0 & 0\\
1 & 2 & 2\\
f & -(1-f) & f
\end{bmatrix}
\begin{bmatrix}
e & 0 & 0\\
0 & f & 0\\
0 & 0 & -e-f
\end{bmatrix}
\begin{bmatrix}
1+e & e-1 & 0\\
e & e-2 & 1\\
e & e-2 & 0
\end{bmatrix}=
\begin{bmatrix}
e & 0 & 0\\
0 & e & f\\
0 & f & 0
\end{bmatrix}.
$$

Using Proposition \ref{prop:aux_rank_id_1}, we have the first equality in the assertion. The second equality follows by exchanging the roles of $e$ and $f$.
\end{proof}

\begin{prop}
\label{prop:commutator_rank}
\textsl{
Let $\sR$ be a ring. For idempotents $e, f$ in $\mathcal{R}$, we have $$\mathrm{rank}_u(ef-fe) \odot \mathrm{rank}_u(1) = \mathrm{rank}_u(e-f) \odot \mathrm{rank}_u(1-e-f).$$ 
}
\end{prop}
\begin{proof}
Note that $(e-f)(e+f-1) = ef-fe.$ 
\vskip 0.1in

\noindent \underline{{\it Step 1.}} Use Lemma \ref{lem:invertible_M2} with the appropriate choices of $x, y$ to identify the invertible matrices for Step 2 below.
\vskip 0.1in

\noindent \underline{{\it Step 2.}} 
\begin{align*}\begin{bmatrix}
0 & 1 \\
1 & 1-2e
\end{bmatrix}
\begin{bmatrix}
1 &  e+f-1 \\
e-f & 0
\end{bmatrix}
\begin{bmatrix}
1 & 0\\
1-2f & 1
\end{bmatrix} &= 
\begin{bmatrix}
e-f & 0\\
0 & 1-e-f
\end{bmatrix}\\
\begin{bmatrix}
e-f & -1 \\
1 & 0
\end{bmatrix}
\begin{bmatrix}
1 &  e+f-1 \\
e-f & 0
\end{bmatrix}
\begin{bmatrix}
1-e-f & 1\\
1 & 0
\end{bmatrix} &= 
\begin{bmatrix}
(e-f)(e+f-1) & 0\\
0 & 1
\end{bmatrix}
\end{align*}

We conclude that
\begin{align*}
\mathrm{rank}_u \Big( 
\begin{bmatrix}
1 &  e+f-1 \\
e-f & 0
\end{bmatrix} \Big) &= \mathrm{rank}_u(ef-fe) \odot \mathrm{rank}_u(1)\\
&= \mathrm{rank}_u(e-f) \odot \mathrm{rank}_u(1-e-f)
\end{align*}
\end{proof}

\subsection{Examples of rank identities in ranked rings}

\begin{lem}
\label{lem:cochran_rank}
\textsl{
Let $\sR$ be a $(\sM, \rho)$-ranked ring. Then for $x \in \mathcal{R}$, we have $$\rho(x) + \rho(1-x) = \rho(1) + \rho(x - x^2).$$ In addition, if $\sM$ is cancellative and $\rho$ is nondegenerate, then $\rho(x) + \rho(1-x) = \rho(1)$ if and only if $x$ is an idempotent.
}
\end{lem}
\begin{proof}
For an indeterminate $t$, note that $t \Z[t] + (1-t) \Z[t] = \Z[t]$. By Lemma \ref{lem:x_1-x}, in $\fU \fR (\Z[t])$ we have $$\mathrm{rank}_u(t) \odot \mathrm{rank}_u(1-t) = \mathrm{rank}_u(t-t^2) \odot \mathrm{rank}_u(1).$$

Since there is an evaluation homomorphism, $\sE : \Z[t] \to \sR$ which sends $t \mapsto x$, by Proposition \ref{prop:rank_id_fund}, we have $$\rho(x) + \rho(1-x) = \rho(1) + \rho(x - x^2).$$ 

Using the cancellation property of $\sM$, we have $\rho(x) + \rho(1-x) = \rho(1)$ if and only if $\rho(x-x^2) = 0$. Since $\rho$ is nondegenerate, we conclude that $x - x^2 = 0$.
\end{proof}

\begin{prop}
\label{prop:diff_rank_id_con}
\textsl{
Let $\sR$ be a $(\sM, \rho)$-ranked ring. If $\sM$ is cancellative, then for idempotents $e, f$ in $\mathcal{R}$, we have 
\begin{itemize}
\item[(i)] $\rho(e-f) = \rho \big( e(1-f) \big) + \rho \big( (1-e)f \big) = \rho \big( e(1-f) \big) + \rho \big( f(1-e) \big)$
\item[(ii)] $\rho(e-f) + \rho(f) = \rho(e) + \rho \big( f(1-e) \big) + \rho \big( (1-e)f \big)$
\end{itemize}
}
\end{prop}
\begin{proof}
Since $\sM$ is cancellative, the assertions follow from Proposition \ref{prop:rank_id_fund} and Lemma \ref{lem:diff_rank_id}.
\end{proof}

\begin{cor}
\label{cor:diff_rank_id_1}
\textsl{
Let $G^+$ be the positive cone of a partially ordered abelian group $G$, and $\sR$ be a $(G^+, \rho)$-ranked ring. Let $e, f$ be idempotent elements in $\sR$. Then we have $\rho(e-f) + \rho(f) = \rho(e)$ if and only if $\rho \big( f(1-e) \big) = \rho \big( (1-e)f \big)=0$. In addition, if $\rho$ is nondegenerate, then $\rho(e-f) + \rho(f) = \rho(e)$ if and only if $f = ef = fe = efe$.
}
\end{cor}
\begin{proof}
From Proposition \ref{prop:diff_rank_id_con}, (ii), and the cancellation property of $G$, we have $\rho(e-f) + \rho(f) = \rho(e)$ if and only $\rho \big( f(1-e) \big) + \rho \big( (1-e)f \big) = 0.$ Since $G$ is positive, we conclude that $\rho \big( f(1-e) \big) + \rho \big( (1-e)f \big) = 0$ if and only if $\rho \big( f(1-e) \big) = \rho \big( (1-e)f \big)$.
\end{proof}

\begin{cor}
\label{cor:app_ortho_rank-ring}
\textsl{
Let $G^+$ be the positive cone of a partially ordered abelian group $G$. Let $\sR$ be a $(G^+, \rho)$-ranked ring. For idempotents $e, f$ in $\mathcal{R}$, if $ef = fe = 0$, then $\rho(e+f) = \rho(e) + \rho(f)$.
}
\end{cor}
\begin{proof}
If $ef=fe =0$, then $e+f$ is an idempotent and $(e+f)f = f = f(e+f)$. Thus by Corollary \ref{cor:diff_rank_id_1}, $\rho(e) + \rho(f) = \rho \big( (e+f)-f \big) +\rho(f) = \rho(e+f)$.
\end{proof}

\begin{cor}
\label{cor:diff_rank_id_2}
Let $\sR$ be a $(\sM, \rho)$-ranked ring with $\sM$ being cancellative. For idempotents $e, f$ in $\mathcal{R}$, we have $\rho(1-e-f) + \rho(e) + \rho(f) = \rho(ef) + \rho(fe)  + \rho(1).$
\end{cor}
\begin{proof}
Using Proposition \ref{prop:diff_rank_id_con}, (ii), for the idempotents $1-e, f$, we observe that $\rho(1-e-f) + \rho(f) = \rho(1-e) + \rho(fe) + \rho(ef)$. From Lemma \ref{lem:cochran_rank}, we have $\rho(e) + \rho(1-e) = \rho(1)$ and hence $\rho(1-e-f) + \rho(e) + \rho(f) = \rho(ef) + \rho(fe)  + \rho(1)$.
\end{proof}

\begin{prop}
\label{prop:diff_rank_id_3}
\textsl{
Let $G^+$ be the positive cone of a partially ordered abelian group $G$. Let $\sR$ be a $(G^+, \rho)$-ranked ring with $G$ where $\rho$ is nondegenerate. Let $e_1, e_2, e_3$ be idempotents in $\mathcal{R}$ such that $e_1 + e_2 + e_3 = 1$. Then $\rho(e_1) + \rho(e_2)+ \rho(e_3) = \rho(1)$ if and only if $e_1, e_2, e_3$ are mutually orthogonal.
}
\end{prop}
\begin{proof}
Using Corollary \ref{cor:diff_rank_id_2} and the cancellation property of $G$, we conclude that $\rho(e_1) + \rho(e_2) + \rho(e_3) = \rho(1-e_2-e_3) + \rho(e_2)+ \rho(e_3) = \rho(1)$ if and only if $\rho(e_2e_3) + \rho(e_3e_2) = 0$. Since $G$ is positive and $\rho$ is nondegenerate, $\rho(e_2 e_3) + \rho(e_3 e_2) = 0$ if and only if $e_2 e_3 = e_3 e_2 = 0$. Thus $e_1, e_2, e_3$ are mutually orthogonal.
\end{proof}

\begin{prop}
\label{prop:subadd_lem}
\textsl{
Let $\sR$ be a $(\sM, \rho)$-ranked ring such that $\sM$ is cancellative and the central element $2$ is invertible in $\sR$. For idempotents $e, f \in \sR$, we have the following rank identity: $$\rho(e+f) = \rho(f) + \rho \big( (1-f)e(1-f) \big).$$
}
\end{prop}

\begin{prop}
\label{prop:comm_anticomm}
Let $\sR$ be a $\rho$-ranked ring. For idempotents $e, f$ in $\mathcal{R}$, we have the following rank identities :
\begin{itemize}
\item[(i)] $\rho(ef-fe) + \rho(1) = \rho(e-f) + \rho(1-e-f)$,
\item[(ii)] $\rho(ef+fe) + \rho(1) = \rho(e+f) + \rho(1-e-f)$.
\end{itemize}
\end{prop}
\begin{proof}
\noindent (i) Follows from Proposition \ref{prop:commutator_rank} and Proposition \ref{prop:rank_id_fund}.
\vskip 0.05in

\noindent (ii) As $ef + fe = (e+f)^2 - (e+f)$, the result follows from Lemma \ref{lem:cochran_rank}.
\end{proof}

\begin{lem}
\label{lem:rank_id_ex1}
\textsl{
Let $\mathcal{R}$ be a $(\sM, \rho)$-ranked ring where $\sM$ is cancellative. For $x, y$ in $\mathcal{R}$, we have the following rank identities
\begin{itemize}
\item[(i)] $\rho(1+xy) = \rho(1+yx)$,
\item[(ii)] $\rho(x) + \rho(1+yx) = \rho(x + xyx) + \rho(1).$
\end{itemize}
}
\end{lem}
\begin{proof}
Follows from Proposition \ref{prop:free_assoc_id} and Theorem \ref{thm:rank_non_pol_id}.
\end{proof}

\section{Applications}
\label{sec:app}

In this section, we illustrate various applications of abstract rank identities to rank-functions on von Neumann regular rings, and in particular, Murray-von Neumann algebras.

\subsection{Von Neumann regular ranked-rings}

\begin{prop}
\label{prop:idem_app_ranked-ring}
\textsl{
Let $\sR$ be a ring and $x, y \in \sR$.
\begin{itemize}
    \item[(i)] If $x \sR = y \sR$, or $\sR x = \sR y$, then $\mathrm{rank}_u(x) = \mathrm{rank}_u(y).$
    \item[(ii)] Let $e$ be an idempotent in $\sR$ such that $e \sR \subseteq x\sR$. Then $$\mathrm{rank}_u(e) \odot \mathrm{rank}_u \big( (1-e)x \big) = \mathrm{rank}_u (x).$$
    \item[(iii)] Let $e$ be an idempotent in $\sR$ such that $\sR e \subseteq \sR x$. Then $$\mathrm{rank}_u(e) \odot \mathrm{rank}_u \big( x(1-e) \big) = \mathrm{rank}_u (x).$$
\end{itemize}
}
\end{prop}
\begin{proof}
\noindent (i) If $x \sR = y \sR$,  there are elements $a, b \in \sR$ such that $x = ya$ and $y = xb$. From Proposition \ref{prop:free_assoc_id}, (iii), we have
$$\mathrm{rank}_u(x) = 
\mathrm{rank}_u \Big(
\begin{bmatrix}
x & xb\\
0 & 0\\
\end{bmatrix}
\Big) = 
\mathrm{rank}_u \Big(
\begin{bmatrix}
y & ya\\
0 & 0\\
\end{bmatrix}
\Big) = 
\mathrm{rank}_u(y).
$$

If $\sR x = \sR y$, the assertion follows similarly using Proposition \ref{prop:free_assoc_id}, (iv).
\vskip 0.1in

\noindent (ii) Let $e \in x \sR$ so that $e = xy$ for some $y \in \sR$. From Lemma \ref{lem:fund_id}, (iv), and Proposition \ref{prop:free_assoc_id}, (iii), we observe that
\begin{equation*}
\label{eqn:rank_ineq_1}
\mathrm{rank}_u(e) \odot \mathrm{rank}_u \big( (1-e)x \big) = \mathrm{rank}_u \Big(
\begin{bmatrix}
e & x\\
0 & 0
\end{bmatrix} \Big) =
\mathrm{rank}_u \Big(
\begin{bmatrix}
xy & x\\
0 & 0
\end{bmatrix} \Big) = \mathrm{rank}_u(x).
\end{equation*}
\vskip 0.1in

\noindent (iii) Let $e \in \sR x$ so that $e = yx$ for some $y \in \sR$. From Lemma \ref{lem:fund_id}, (iii), and Proposition \ref{prop:free_assoc_id}, (iv), we observe that
\begin{equation*}
\label{eqn:rank_ineq_2}
\mathrm{rank}_u(e) \odot \mathrm{rank}_u \big( x(1-e) \big) =
\mathrm{rank}_u \Big(
\begin{bmatrix}
e & 0\\
x & 0
\end{bmatrix} \Big) =
\mathrm{rank}_u \Big(
\begin{bmatrix}
yx & 0\\
x & 0
\end{bmatrix} \Big) = \mathrm{rank}_u(x).
\end{equation*}
\end{proof}

\begin{cor}
\label{cor:idem_app_ranked-ring}
\textsl{
Let $G^{+}$ be the positive cone of a partially ordered abelian group $G$. Let $\sR$ be a von Neumann regular $(G^{+}, \rho)$-ranked ring.
\begin{itemize}
\item[(i)] Let $x, y \in \sR$ such that $x \sR \subseteq y\sR$. Then $\rho(x) \le \rho(y)$. If $\rho$ is nondegenerate and $\rho(x) = \rho(y)$, then $x \sR = y \sR$.
\item[(ii)] Let $x, y \in \sR$ such that $\sR x \subseteq \sR y$. Then $\rho(x) \le \rho(y)$. If $\rho$ is nondegenerate and $\rho(x) = \rho(y)$, then $\sR x = \sR y$.
\end{itemize}
}
\end{cor}
\begin{proof}
\noindent (i) Let $e, f$ be idempotents in $\sR$ such that $x \sR = e \sR, y \sR = f \sR$. From Proposition \ref{prop:idem_app_ranked-ring}, (i), we have $\rho(x) = \rho(e), \rho(y) = \rho(f)$. From Proposition \ref{prop:idem_app_ranked-ring}, (ii), we have $\rho(f) - \rho(e) = \rho \big( (1-e)f \big) \ge 0$. If $\rho$ is nondegenerate and $\rho(e) = \rho(f)$, we have $(1-e)f = 0$ which implies that $f \sR \subseteq e \sR$, or equivalently, $y \sR \subseteq x \sR$.
\vskip 0.1in

\noindent (ii) Let $e, f$ be idempotents in $\sR$ such that $\sR x =  \sR e, \sR y = \sR f$. From Proposition \ref{prop:idem_app_ranked-ring}, (i), we have $\rho(x) = \rho(e), \rho(y) = \rho(f)$. From Proposition \ref{prop:idem_app_ranked-ring}, (iii), we have $\rho(f) - \rho(e) = \rho \big( f(1-e) \big) \ge 0$. If $\rho$ is nondegenerate and $\rho(e) = \rho(f)$, we have $f(1-e) = 0$ which implies that $\sR f \subseteq \sR e$, or equivalently, $\sR y \subseteq \sR x$.
\end{proof}

\begin{prop}
\label{prop:vNreg_rank}
\textsl{
Let $G^{+}$ be the positive cone of a partially ordered abelian group $G$. Let $\sR$ be a von Neumann regular $(G^{+}, \rho)$-ranked ring. Let $n \in \N$.
\begin{itemize}
\item[(i)] For $x, y \in M_n(\sR)$,  we have $\rho(xy) \le \rho(x)$ and $\rho(xy) \le \rho(y)$. 
\item[(ii)] Let $\rho$ be nondegenerate. Then for $x, y \in M_n(\sR)$ we have the rank equality $\rho(xy) = \rho(x)$ if and only if $xy\sR = x \sR$. Similarly, we have the rank equality $\rho(xy) = \rho(y)$ if and only if $\sR xy = \sR y$.
\item[(iii)] For every pair of mutually orthogonal idempotents $e, f \in M_n(\sR)$, we have $\rho(e+f) = \rho(e) + \rho(f)$.
\end{itemize}
}
\end{prop}
\begin{proof}
By restricting $\rho$ to $M_{\mathrm{fin}}(M_n(\sR)) \subseteq M_{\mathrm{fin}}(\sR)$, we have a $G ^{+}$-valued rank-system for the von Neumann regular ring $M_n(\sR)$. Furthermore, if $\rho$ is nondegenerate, then the restriction of $\rho$ also defines a nondegenerate rank-system for $M_n(\sR)$. Thus we need only prove the assertions in the case of $n=1$.

Let $x, y \in \sR$. Since $xy \sR \subseteq x \sR$ and $\sR xy \subseteq \sR y$, assertions (i)-(ii), follow from Corollary \ref{cor:idem_app_ranked-ring}. Assertion (iii) follows directly from Corollary \ref{cor:app_ortho_rank-ring}.
\end{proof}

\begin{thm}
\label{thm:rank_sys_ring_conv}
\textsl{
Let $\sR$ be a von Neumann regular $(\R ^{+}, \rho)$-ranked ring. Let $\fr ^{(1)}$ denote the restriction of $\rho$ to $\sR$. Then $\fr ^{(1)}$ is a rank-function on $\sR$. Moreover, $\rho(X) = \big( (\delta _k ^*)^{-1} \fr ^{(1)} \big)(X)$ for $X \in M_k(\sR)$, where $\delta _k ^*$ is as defined in the statement of Theorem \ref{thm:pushout_rank}. 
}
\end{thm}
\begin{proof}
Let $\fr ^{(k)}$ denote the restriction of $\rho$ to $M_k(\sR)$. By Proposition \ref{prop:vNreg_rank}, $\fr ^{(k)}$ is a rank-function on the von Neumann regular ring $M_k(\sR)$. Moreover, for $x \in \sR$, we have $\fr ^{(k)}(\odot_{i=1}^k x) = \rho(\odot_{i=1}^k x) = n \rho(x) = n \fr ^{(1)}(x).$ From Theorem \ref{thm:pushout_rank}, we conclude that $\fr ^{(k)} = (\delta _k ^*)^{-1} \fr ^{(1)}$.
\end{proof}

\begin{remark}
\label{rmrk:rank_sys_ring}
From Theorem \ref{thm:rank_sys_ring} and Theorem \ref{thm:rank_sys_ring_conv}, we observe that there is a natural one-to-one correspondence between rank-functions on $\sR$ and $\R ^{+}$-valued rank-systems on $\sR$.
\end{remark}

\begin{prop}[subadditivity of rank]
\label{prop:vNreg_subadd}
\textsl{
Let $G^{+}$ be the positive cone of a partially ordered abelian group $G$. Let $\sR$ be a von Neumann regular $(G^{+}, \rho)$-ranked ring. Then for $x, y \in \sR$, we have
$$\rho(x+y) \le \rho(x) + \rho(y).$$
}
\end{prop}
\begin{proof}
For $x, y \in\sR$, we observe that 
$$
\begin{bmatrix}
x + y & 0\\
0 & 0
\end{bmatrix} = 
\begin{bmatrix}
1 & 1\\
0 & 0
\end{bmatrix}
\begin{bmatrix}
x & 0\\
0 & y
\end{bmatrix}
\begin{bmatrix}
1 & 0\\
1 & 0
\end{bmatrix}.
$$
Applying Proposition \ref{prop:vNreg_rank}, (i), (in the context of $M_2(\sR)$), we conclude that 
$$
\rho \Big(
\begin{bmatrix}
x + y & 0\\
0 & 0
\end{bmatrix} \Big) \le 
\rho \Big( 
\begin{bmatrix}
x & 0\\
0 & y
\end{bmatrix}\Big).
$$
Thus $\rho(x+y) \le \rho(x) + \rho(y).$ 
\end{proof}

\begin{thm}
\label{thm:vna_rank_sub}
\textsl{
Let $\mathscr{R}$ be a finite von Neumann algebra. For $x, y \in \afr$, we have $$\fr _c(x + y) \le \fr _c(x) + \fr _c(y),$$
with equality if and only if $R(x) \wedge R(y) = R(x^*) \wedge R(y^*) = 0$.
}
\end{thm}
\begin{proof}
Although the subadditivity of $\fr _c$ immediately follows from Proposition \ref{prop:vNreg_subadd}, we give an alternate proof below which also naturally yields the conditions for equality.

Clearly $R(x+y) \le R(x) \vee R(y)$. Note that $e := R(x), f := R(y)$ are projections in $\mathscr{R}$. By Lemma \ref{lem:range_proj_id}, (iii), we have $e \vee f = R(e+f)$ and hence $R(x+y) \le R(e+f)$. From Proposition \ref{prop:subadd_lem} and Proposition \ref{prop:rank_prop}, (v)--(vi), we have $$\fr _c(x+y) \le \fr _c(e+f) = \fr _c \big( e(1-f) \big) + \fr_c(f) \le \fr _c(e) + \fr _c(f) = \fr _c(x) + \fr _c(y).$$  
 
Next we turn our attention to the case when equality holds. If $\fr _c(x+y) = \fr _c(x) + \fr _c(y)$, then we have $\fr _c(e) = \fr _c \big( e(1-f) \big)$.  By Lemma \ref{lem:range_proj_id}, (iv), note that  $\fr(e) = \fr(e - e \wedge f)$ and thus by Proposition \ref{prop:rank_prop}, (iv), we conclude that $e \wedge f = R(x) \wedge R(y) = 0.$ Similarly, as by Proposition \ref{prop:rank_prop}, (v),  $\fr _c(x^*+y^*) = \fr _c(x+y) = \fr _c(x) + \fr _c(y) = \fr _c(x^*) + \fr _c(y^*)$, we have $R(x^*) \wedge R(y^*) = 0$.

For the converse, let us assume that $R(x) \wedge R(y) = 0$ and $R(x^*) \vee R(y^*) = 0$. As $e \wedge f = 0$, by Lemma \ref{lem:range_proj_id}, (iv), we have $\fr _c \big( e(1-f) \big) = \fr _c(e)$ and hence $\fr _c(e + f) = \fr _c(e) + \fr _c(f) = \fr _c(x) + \fr _c(y)$. Let $p$ be a projection in $\mathscr{R}$ such that $(x^*+y^*)p = 0$ or equivalently, $x^*p = -y^*p$. Since $R(x^*) \wedge R(y^*) = 0$, we conclude that $x^* p = y^*p = 0$, or equivalently, $p \le N(x^*)\wedge N(y^*)$. Hence $N(x^*+y^*) \le N(x^*) \wedge N(y^*)$. It is straightforward to see that $N(x^*) \wedge N(y^*) \le N(x^*+y^*)$. Thus $N(x^*+y^*) = N(x^*) \wedge N(y^*)$ and using Lemma \ref{lem:range_proj_rel}, (i) and Lemma \ref{lem:range_proj_id}, (iii), we arrive at $R(x+y) = R(x) \vee R(y) = e \vee f$.  Hence $\fr _c(x+y) = \fr _c(e \vee f) = \fr _c(e + f) = \fr _c(x) + \fr _c(y).$
\end{proof}

\subsection{Generalized Cochran's theorem}

\begin{thm}[Generalized Cochran's Theorem]
\label{thm:cochran}
\textsl{
Let $G^+$ be the positive cone of a partially ordered abelian group $G$. Let $\sR$ be a $(G^+, \rho)$-ranked ring where $\rho$ is nondegenerate and subadditive. Let $a_1, \ldots, a_n \in \sR$ and $e$ be an idempotent in $\sR$ such that $\sum_{i=1}^n a_i = e$. Then $\sum_{i=1}^n \rho(a_i) = \rho(e)$ if and only if the $a_i$'s are mutually orthogonal idempotents, that is, $a_i a_j  = \delta_{ij}a_i,$ for $1 \le i, j \le n$.
}
\end{thm}
\begin{proof}
Let $a_1, a_2, \ldots, a_n$ be mutually orthogonal idempotents in $\sR$ so that $e := \sum_{i=1}^n a_i$ is an idempotent. It is easy to see that $f := \sum_{i=2}^n a_i$ is an idempotent, and $a_1, f$ are mutually orthogonal. From Corollary \ref{cor:app_ortho_rank-ring}, we conclude that $\rho(a_1) + \rho(f) = \rho(e)$. Similarly $\rho(a_2) + \rho(\sum_{i=3}^n a_i) = \rho(f)$. Proceeding inductively, we conclude that $\sum_{i=1}^n \rho (a_i) = \rho (e)$.

Next we prove the converse inductively. Consider the first non-trivial base case of $n=2$. Let $a_1, a_2 \in \sR$ such that $e := a_1 + a_2$ is an idempotent. Let $\rho(a_1) + \rho(a_2) = \rho(a_1 + a_2)$. With $f := 1-e$, by Lemma \ref{lem:cochran_rank}, observe that $\rho(a_1) + \rho(a_2) + \rho(f) = \rho(e) + \rho(f) = \rho(1)$. From the sub-additivity of the rank (see Proposition \ref{prop:vNreg_subadd}), we see that $$\rho(1) \le \rho(a_1) + \rho(1-a_1) \le \rho(a_1) + \rho(a_2) + \rho(f) = \rho(1),$$ which implies that $\rho(a_1) + \rho(1 - a_1) = \rho(1)$. Thus by Lemma \ref{lem:cochran_rank}, $a_1 ^2 = a_1$ and by a similar argument, $a_2 ^2 = a_2$. As $\rho(e - a_1) + \rho(a_1) = \rho(e)$, by Corollary \ref{cor:diff_rank_id_1}, we have $e a_1 = a_1 e = a_1 \Longrightarrow a_1 ^2 + a_2 a_1 = a_1 ^ 2 + a_1 a_2 = a_1 \Longrightarrow a_1 a_2 = a_2 a_1 = 0$. We conclude that $a_1, a_2$ are mutually orthogonal idempotents.

For $n \ge 3$, let us assume that the assertion is true for any collection of $n-1$ operators satisfying the given conditions. For $a_1, \ldots, a_n$ in $\sR$, let $e := \sum_{i=1}^n a_i$ be an idempotent  and $\rho(e) = \sum_{i=1}^n \rho(a_i)$. Define $a := a_1 + \cdots + a_{n-1}$ and $b := a_{n}$. Note that $a+b = e$ and from the sub-additivity of the rank, $\rho(e) \le \rho(a) + \rho(b) \le \sum_{i=1}^n \rho(a_i) = \rho(e).$ Thus $\rho(a) + \rho(b) = \rho(e)$. From the base case of $n=2$, we have $a, b$ are idempotents and $ab=ba=0$. Further $\rho(a) = \sum_{i=1}^{n-1} \rho( a_i)$. By the induction hypothesis, we have $a_i a_j = \delta_{ij} a_i, 1\le i, j \le n-1$. Note that we can choose $b$ to be any of the $a_i$'s. Thus we conclude that $a_i a_j = \delta_{ij} a_j, 1\le i, j \le n$.
\end{proof}

\begin{cor}
\label{cor:cochran_vNreg}
\textsl{
Let $G^+$ be the positive cone of a partially ordered abelian group $G$. Let $\sR$ be a von Neumann regular $(G^+, \rho)$-ranked ring where $\rho$ is nondegenerate. Let $a_1, \ldots, a_n \in \sR$ and $e$ be an idempotent in $\sR$ such that $\sum_{i=1}^n a_i = e$. Then $\sum_{i=1}^n \rho(a_i) = \rho(e)$ if and only if the $a_i$'s are mutually orthogonal idempotents, that is, $a_i a_j  = \delta_{ij}a_i,$ for $1 \le i, j \le n$.
}
\end{cor}
\begin{proof}
By Proposition \ref{prop:vNreg_subadd}, $\rho$ is subadditive. The assertion follows from Theorem \ref{thm:cochran}.
\end{proof}

\begin{cor}
\label{cor:cochran_MvN}
\textsl{
Let $\mathscr{R}$ be a finite von Neumann algebra. Let $a_1, \ldots, a_n \in \afr$ and $e$ be an idempotent in $\afr$ such that $\sum_{i=1}^n a_i = e$. Then $\sum_{i=1}^n \fr _c(a_i) = \fr _c(e)$ if and only if the $a_i$'s are mutually orthogonal idempotents, that is, $a_i a_j  = \delta_{ij}a_i,$ for $1 \le i, j \le n$.
}
\end{cor}
\begin{proof}
Note that $\afr$ is a von Neumann regular ring, and $\fr _c$ is a nondegenerate rank-function on $\afr$ taking values in the positive part of the center of $\mathscr{R}$. The assertion follows from Corollary \ref{cor:cochran_vNreg}.
\end{proof}

\begin{thm}
\label{thm:sum_of_idem_fva}
\textsl{
Let $\mathscr{R}$ be a finite von Neumann algebra. For idempotents $e_1, e_2, \ldots, e_n$ in $\mathscr{R}$, $e_1 + e_2 + \cdots + e_n$ is idempotent if and only if the $e_i$'s are mutually orthogonal, that is, $e_i e_j = \delta_{ij} e_i, 1\le i, j \le n$.
}
\end{thm}
\begin{proof}
Let $\sA := \{ x \in \mathscr{R} : \fr _c(x) = \tau(x) \}$. Clearly $\sA$ contains every projection in $\mathscr{R}$ and their similarity orbits. By \cite[Lemma 16]{kaplansky_mod}, every idempotent in $\mathscr{R}$ is contained in the similarity orbit of a projection in $\mathscr{R}$. Thus for every idempotent $f \in \mathscr{R}$, we have $\fr _c(f) = \tau(f)$. 

Let $e := e_1 + \cdots + e_n$ be an idempotent. From the discussion in the preceding paragraph, we have $\fr _c(e) = \tau(e) = \tau(e_1) + \cdots + \tau(e_n) = \fr _c(e_1) + \cdots + \fr _c(e_n)$. By Theorem \ref{thm:cochran}, we conclude that the $e_i$'s are mutually orthogonal. For the converse, note that $(e_1 + \cdots + e_n)^2 = \sum_{i=1}^n \sum_{j=1}^n e_i e_j.$ 
\end{proof}

\begin{remark}
Let $e_1, e_2$ be idempotents in a ring not of characteristic $2$, such that $e_1 + e_2$ is an idempotent. Note that $e_1 + e_2 = (e_1 + e_2)^2 = e_1 ^2 + e_2 ^2 + e_1 e_2 + e_2 e_1 = e_1 + e_2 + e_1 e_2 + e_2 e_1$. We have
\begin{align*}
    e_1 e_2 = -e_2 e_1 &\Longrightarrow e_1 e_2 (1 - e_1) = 0\\ 
    & \Longrightarrow e_1 e_2 e_1 = e_1 e_2\\ 
    & \Longrightarrow (e_1 e_2)^2 = e_1 e_2.
\end{align*}
Thus $e_1e_2$ is an idempotent and by a symmetric argument $-e_1 e_2 = e_2 e_1$ is an idempotent. Hence $e_1 e_2 = -e_1 e_2 = 0$ and similarly $e_2 e_1=0$. In other words, $e_1$ and $e_2$ are mutually orthogonal.

In light of the above discussion, it is natural to wonder whether the content of Theorem \ref{thm:sum_of_idem_fva} is a purely algebraic fact and that perhaps the introduction of the rank $\fr _c$ is an artifice. As a counterexample, we note that one can find five idempotent operators acting on an infinite-dimensional complex Hilbert space that are not mutually orthogonal and whose sum is $0$, which is an idempotent (see \cite[Example 3.1]{bart_zero_sum}).
\end{remark}

\subsection{The Heisenberg-von Neumann puzzle}

In \cite[Corollary 5.4]{nayak_matrix_mvna}, the rank identity in Proposition \ref{prop:free_assoc_id}, (i), has been used to prove Proposition \ref{prop:heisenberg_drange} below in the setting of $II_1$ factors. It is straightforward to extend the result to general finite von Neumann algebras via $\fr _c$. 

\begin{prop}
\label{prop:heisenberg_drange}
\textsl{
Let $\mathscr{R}$ be a finite von Neumann algebra. Let $p, q \in \afr$ such that $qp-pq = \iu 1$. Then for all $\lambda \in \C$, $p - \lambda 1$ and $q - \lambda 1$ are invertible in $\afr$, that is, the respective point spectra of $p$ and $q$ are empty.
}
\end{prop}

As mentioned in \S \ref{subsec:op_alg_res_overview}, Proposition \ref{prop:heisenberg_drange} makes it amply clear that any strategy towards the resolution of the Heisenberg-von Neumann puzzle must involve the study of similarity orbits in $\afr$.

\section{Future Work}

In this section, we note some problems to direct future work on this topic.

\begin{prob}
\textsl{
Let $\K$ be a field. Characterize the $\fU \fR$-monoid of $\K \langle t_1, t_2 \rangle$. More generally, for $n \in \N$, characterize the $\fU \fR$-monoid of the free associative algebra in $n$ indeterminates over $\K$. 
}
\end{prob}

\begin{prob}
\textsl{
Let $\K$ be a field. Characterize the $\fU \fR$-monoid of $\K \langle t_1, t_2\rangle/(t_1 - t_1 ^2, t_2 - t_2 ^2)$, the universal associative $\K$-algebra generated by two idempotents. More generally, for $n \in \N$, characterize the $\fU \fR$-monoid of the universal associative $\K$-algebra generated by $n$ idempotents.
}
\end{prob}

\begin{prob}
\textsl{
Let $X$ be a compact Hausdorff space. Characterize the $\fU \fR$-monoid of $C(X)$, the ring of continuous complex-valued functions on $X$. (Of particular interest are compact subsets of $\C$ for applications, via the continuous function calculus, to rank identities involving normal operators.)
}
\end{prob}

\begin{prob}
\textsl{
Let $G^{+}$ be the positive cone of a partially ordered abelian group $G$. Let $\sR$ be a von Neumann regular $(G^{+}, \rho)$-ranked ring. For $x, y \in \sR$, with the help of Proposition \ref{prop:vNreg_rank}, (ii), it can be shown that if $\rho$ is nondegenerate and $\rho(x+y) = \rho(x) + \rho(y)$, then $x \sR + y \sR = (x+y)\sR$ and $\sR x + \sR y = \sR (x+y)$. Does the converse hold?
}
\end{prob}

\section{Acknowledgments}
I would like to express my gratitude to Gregory Grant for his lectures on linear statistical models which introduced me to Cochran's theorem that led me down a path which culminated in this paper. I would also like to thank Max Gilula, Dick Kadison, Ben Pollak, K.\ V.\ Shuddhodan, and B.\ Sury for helpful discussions and suggestions. A special note of thanks to an anonymous referee whose suggestions greatly helped improve the presentation of the material.

\bibliographystyle{plain}
\bibliography{reference}

\end{document}